\documentclass[preprint,12pt]{elsarticle}

\usepackage{graphicx}%
\usepackage{multirow}%
\usepackage{amsmath,amssymb,amsfonts,mathabx}%
\usepackage{amsthm}%
\usepackage{mathrsfs}%
\usepackage[title]{appendix}%
\usepackage{xcolor}%
\usepackage{textcomp}%
\usepackage{manyfoot}%
\usepackage{booktabs}%
\usepackage{algorithm}%
\usepackage{algorithmicx}%
\usepackage{algpseudocode}%
\usepackage{listings}%
\usepackage{array}
\usepackage{float} 
\usepackage{multirow} 
\usepackage{epstopdf} 
\usepackage{enumitem}

\makeatletter
\newcommand{\thickhline}{%
	\noalign {\ifnum 0=`}\fi \hrule height 1pt
	\futurelet \reserved@a \@xhline
}

\newcolumntype{L}{>{\centering\arraybackslash}m{3cm}} 

\newcommand{\etal}{\mbox{\emph{et al.\ }}} 
 
\newcommand{\E}{\mathcal{E}}

\DeclareMathOperator{\diag}{diag}

\newlength\myindent
\newtheorem{theorem}{Theorem}[section]

\newtheorem{lemma}[theorem]{Lemma}
\newtheorem{proposition}[theorem]{Proposition}

\newtheorem{example}[theorem]{Example} 
\newtheorem{assumption}[theorem]{Assumption}
\newtheorem{remark}{Remark}

\journal{Mathematics and Computers in Simulation}

\begin{document}

\begin{frontmatter}

\title{Analysis and Efficient Sylvester-Based Implementation of a Dimension-Split ETD2RK Scheme for Multidimensional Reaction–Diffusion Equations}

\author[kfupm,crac]{Ibrahim O.~Sarumi \corref{cor1}}
\ead{ibrahim.sarumi@kfupm.edu.sa}

\address[kfupm]{Department of Mathematics, King Fahd University of Petroleum \& Minerals,
	Dhahran, Saudi Arabia}
\address[crac]{Center for Refining and Advanced Chemicals, King Fahd University of Petroleum \& Minerals, 
	Dhahran, Saudi Arabia}
\cortext[cor1]{Corresponding author}

\begin{abstract}
We propose and analyze a second-order, dimension-split exponential time differencing Runge--Kutta scheme (ETD2RK-DS) 
for multidimensional reaction--diffusion equations in two and three spatial dimensions. Under mild assumptions on the 
nonlinear source term, we establish uniform stability bounds and prove second-order temporal convergence for the 
underlying dimension-split scheme.

To enable efficient implementation, we employ Pad\'e approximations of the matrix exponential, converting each required 
matrix-exponential--vector product into the solution of a shifted linear system. A convergence analysis of the resulting 
Pad\'e-based ETD2RK-DS formulation is provided. We derive explicit and reproducible tensor-slicing and reshaping algorithms 
that realize the dimension-splitting strategy, decomposing multidimensional systems into collections of independent 
one-dimensional problems. This leads to a reduction of the dominant per-time-step computational cost from 
$\mathcal{O}(m^3)$ to $\mathcal{O}(m^2)$ in two dimensions and from $\mathcal{O}(m^5)$ to $\mathcal{O}(m^3)$ in three 
dimensions when compared with banded LU solvers for the unsplit problem, where $m$ denotes the number of grid points per 
spatial direction.

Furthermore, we develop a Sylvester-equation reformulation of the resulting one-dimensional systems, enabling a highly 
efficient spectral implementation based on reusable eigendecompositions, matrix--vector multiplications, and Hadamard 
divisions. Numerical experiments in two and three dimensions, including a coupled FitzHugh--Nagumo system, confirm the 
second-order temporal accuracy, stability of the underlying scheme, and scalability of the proposed ETD2RK-DS framework, 
as well as the substantial computational advantages of the Sylvester-based implementation over classical LU-based solvers.
\end{abstract}

%


\end{frontmatter}


 
\section{Introduction}\label{sec:Intro}
This paper is concerned with the numerical solution of the multidimensional reaction--diffusion equation
\begin{equation}\label{eq:model}
	u_t(\mathbf{x},t)+\mathcal{A}u(\mathbf{x},t)=f(\mathbf{x},t,u),
	\qquad \mathbf{x}\in\Omega,\ 0~<~t~<~T,
\end{equation} 
posed on a rectangular domain $\Omega \subset \mathbb{R}^2$ or a rectangular box in $\mathbb{R}^3$, and subject to
homogeneous Dirichlet conditions $u(\mathbf{x},t)=0$ on $\partial\Omega$ and prescribed initial data
$u(\mathbf{x},0)=u_0(\mathbf{x})$.
The time invariant operator $\mathcal{A} := -\kappa \Delta  + q$, where the constants $\kappa > 0$ and 
$q \ge 0$, are the diffusivity and potential, respectively. 

Reaction--diffusion systems arise in a wide range of applications, including porous media \cite{Tartakovsky:2019}, 
chemical kinetics \cite{Ipsen:2000}, activator–-inhibitor models in pattern formation \cite{Maini:2004,Bailles:2022}.
For general nonlinear source terms, exact solutions are rarely available, making the development of stable and 
accurate numerical schemes essential.

Spatial discretization of \eqref{eq:model} leads to large-scale systems whose efficient solution presents significant 
computational challenges. In two dimensions, discretization with $m$ grid points per direction yields $\mathcal{O}(m^2)$ 
unknowns and $\mathcal{O}(m^3)$ complexity per solve when banded LU factorization is employed. In three dimensions, this 
worsens to $\mathcal{O}(m^3)$ unknowns and $\mathcal{O}(m^5)$ operations per solve, rendering high-resolution simulations 
prohibitively expensive. This unfavorable scaling has motivated the development of structure-exploiting algorithms that 
reduce computational complexity.

Matrix-equation-based strategies, such as those proposed by Palitta and Simoncini \cite{Palitta:2016} and extended in \cite{Palitta:2021}, 
reformulate multidimensional discretizations as Sylvester equations with coefficient matrices of size $\mathcal{O}(m)$. 
While effective for time-stepping schemes that admit matrix--matrix formulations, such approaches are less convenient for exponential 
integrators, which are naturally expressed in matrix--vector form.
In this context, dimension-splitting strategies \cite{Asante:2020} have been developed that exploit the underlying Kronecker 
structure of the discretized operator to decompose the multidimensional problem into sums of directional operators of the 
same algebraic dimension. This yields matrices 
with improved sparsity and band structure, enabling more efficient banded linear solves.

Reaction--diffusion models are often stiff, limiting the efficiency of explicit schemes. Fully implicit methods provide stability but 
require the solution of nonlinear systems \cite{Madzvamuse:2014}. Implicit--explicit (IMEX) schemes \cite{Ascher:1995,Lakkis:2013} and 
predictor--corrector methods \cite{Rempe:2006} partially alleviate stiffness by treating linear and nonlinear terms separately. 
Exponential time differencing (ETD) schemes address stiffness by integrating the linear part exactly and have proven effective for 
dissipative systems. Various exponential integrators have been proposed, including Lawson methods \cite{Caliari:2024}, multistep 
ETD schemes \cite{Beylkin:1998}, and exponential Rosenbrock methods \cite{Dallerit:2024}. In particular, Cox and Matthews \cite{Cox:2002} 
introduced exponential Runge--Kutta (ETDRK) methods, and their rigorous error analysis has been developed by Hochbruck and Ostermann 
\cite{Hochbruck:2005,Hochbruck:2005b}.

Efficient implementation of ETD schemes in multiple spatial dimensions remains an active area of research. 
Recently, Asante-Asamani \etal\ \cite{Asante:2025} proposed a fourth-order, Pad\'e-based, dimension-split ETDRK 
method and demonstrated strong empirical performance. Their implementation primarily exploits the improved sparsity 
and band structure of the split operators to achieve computational efficiency through the use of 
efficient banded linear solvers, while operating on systems of the original algebraic dimension. This approach 
does not explore the possibility of explicitly reorganizing the discrete system via tensor slicing into smaller 
independent subsystems. Related second-order schemes using rational approximations with real distinct poles have 
been studied in \cite{Asante:2016}, motivated in part by the desire to avoid the complex arithmetic associated 
with Pad\'e approximants.

To the best of our knowledge, rigorous stability and convergence analysis for Pad\'e-based, dimension-split ETDRK schemes, 
together with an explicit and implementation-ready treatment of the matrix slicing and reshaping required to reduce 
multidimensional systems to collections of independent one-dimensional problems, remain underdeveloped. 
Moreover, while several existing approaches either rely primarily on improved sparsity and band structure of split operators 
or avoid Pad\'e approximants altogether in order to bypass complex arithmetic, the efficient solution of the resulting shifted 
linear systems within a Pad\'e-based ETD framework, particularly in the presence of heterogeneous diffusion coefficients, remains 
comparatively unexplored.

In this work, we propose and analyze a second-order, dimension-split exponential time differencing Runge--Kutta scheme 
(ETD2RK--DS) for multidimensional reaction--diffusion equations. For the underlying ETD2RK--DS formulation based on exact 
matrix exponentials, we establish uniform stability and second-order temporal convergence under mild assumptions on the 
nonlinear source term. Pad\'e approximants are then employed for efficient implementation of the fully discrete scheme.  
For this formulation, we provide a rigorous convergence analysis that builds on the structure and stability of the underlying method. 
A second-order dimension-split ETD2RK scheme was previously developed and analyzed in \cite{Chen:2021}, where Krylov 
subspace methods were adopted to approximate the matrix exponential. While the present work shares the dimension-splitting 
framework, the Pad\'e-based formulation considered here leads to a distinct analytical treatment and algorithmic structure, 
and necessitates explicit handling of the shifted linear systems arising from the rational approximation.

From an algorithmic perspective, Pad\'e approximations reduce the action of the matrix exponential on a vector to the solution 
of shifted linear systems. Although dimension splitting provides a natural framework for decomposing a multidimensional 
operator into directional components, at the algebraic level the resulting split operators retain Kronecker structure and act 
on vectors of full multidimensional size. To realize a genuine reduction to independent one-dimensional solves, systematic 
tensor slicing and reshaping of the discrete solution and source terms are required. We derive explicit and reproducible 
matrix-slicing identities that achieve this reduction and lead to substantial complexity savings: from $\mathcal{O}(m^3)$ to 
$\mathcal{O}(m^2)$ per solve in two dimensions and from $\mathcal{O}(m^5)$ to $\mathcal{O}(m^3)$ in three dimensions. This asymptotic 
improvement enables simulations at resolutions that would otherwise be computationally infeasible. 

To efficiently solve the resulting one-dimensional shifted systems, we introduce a Sylvester-equation reformulation that 
enables a spectral implementation based on reusable eigendecompositions. This approach confines complex-valued operations 
associated with Pad\'e approximants to a one-time preprocessing step, see Remark~\ref{rem:algorithms}, and reduces subsequent 
linear solves to real-valued matrix--vector multiplications and Hadamard divisions, thereby mitigating the computational 
overhead typically associated with the complex arithmetic induced by Pad\'e approximants. As a result, the method achieves 
significant computational savings relative to LU-based solvers and remains effective for higher-degree Pad\'e approximants 
as well as systems with heterogeneous diffusivities.

The principal contributions of this work lie in the combination of a rigorous analysis of the Pad\'{e}-based, 
dimension-split exponential time differencing schemes with an explicit tensor-slicing and reshaping 
algorithm, together with a Sylvester-equation reformulation that fundamentally reduces the computational complexity 
of their practical implementation.

The rest of this paper is arranged as follows. Derivation of the fully discrete 
scheme is presented in Section \ref{sec:discretize}. The numerical aspects, 
including the slicing procedure and Sylvester‑equation‑based algorithm are 
presented in Section \ref{Sec:Numerical_Aspects}. Stability and convergence 
analysis are presented in Section \ref{sec:analysis}. In Section \ref{Sec:Numerics} 
we present numerical examples.

\section{Numerical method}\label{sec:discretize}
To construct the fully discrete scheme, we first discretize the model problem
\eqref{eq:model} in space, obtaining a finite-dimensional system of ordinary
differential equations suitable for the application of exponential time
differencing in time. 

\subsection{Spatial Discretization}\label{subsec:spatial_Disc} 
For simplicity of presentation we assume that $\Omega$ is a rectangle
$(x_a,x_b)\times(y_a,y_b)\subset\mathbb{R}^2$ or a cuboid
$(x_a,x_b)\times(y_a,y_b)\times(z_a,z_b)\subset\mathbb{R}^3$. 

We start by approximating the spatial derivatives by the second-order centered difference. In two space dimensions, consider mesh sizes 
\(h_x = (x_b - x_a)/m_x\), \(h_y = (y_b - y_a)/m_y\), and define the set of grid points  
$\Omega^{(2)}_h = \{(x_i,y_j): x_i = x_a + ih_x, i = 1, 2, \dots, m_x - 1, y_j = y_a + j h_y, j = 1, 2, \ldots, m_y-1\}$ of $\Omega \subset \mathbb{R}^2$. 
Similarly, for the three space dimensions, take \(h_z = (z_b - z_a)/m_z\) in the \(z\) direction and define mesh 
$\Omega^{(3)}_h = \{(x_i,y_j,z_k): x_i = x_a + ih_x, i = 1, 2, \dots, m_x - 1, \,\,\,  y_j = y_a + j h_y, j = 1, 2, \ldots, m_y-1, \,\,\,  z_k = z_a + k h_z, 
k = 1, 2, \ldots, m_z-1\}$ of $\Omega \subset \mathbb{R}^3$. 

For \((x_i, y_j) \in \Omega^{(2)}_h \), subject to the regularity assumptions required for the centered difference approximation, we can have  
\begin{equation*}
	u'_{ij}(t) - \kappa \Delta_{h_x}u_{ij}(t) - \kappa \Delta_{h_y}u_{ij}(t) + qu_{ij}(t) + \mathcal{O}(h_x^2) + \mathcal{O}(h_y^2) = f_{ij}(t,u_{ij}(t)).
\end{equation*}  
where 
$$
\Delta_{h_x}u_{ij}(t) = \frac{u(x_{i+1},y_j,t) -2u(x_i,y_j,t) + u(x_{i-1},y_j,t)}{h^2_x},  
$$  
$$
\Delta_{h_y}u_{ij}(t) = \frac{u(x_{i},y_{j+1},t) -2u(x_i,y_j,t) + u(x_{i},y_{j-1},t)}{h^2_y},  
$$
$u_{ij}(t) = u(x_i,y_j,t)$ and $f(x_i,y_j,t,u_{ij}(t)) = f_{ij}(t,u_{ij}(t))$.

Letting $u_{h,ij}(t) \approx u_{ij}(t)$ in the absence of the order terms $\mathcal{O}(h^2_x)$, $\mathcal{O}(h^2_y)$, we can obtain the system
\begin{equation}\label{eq:semidiscrete}
	\mathbf{u}'_{h}(t) + A_h \mathbf{u}_{h}(t) = \mathbf{f}(t,\mathbf{u}_{h}(t)), 
\end{equation} 
where 
$A_h = I_y \otimes A_x + A_y \otimes I_x$, the matrix $A_{x} = \frac{1}{h^2_{x}}[a_{ij}]$ 
is an $(m_x - 1) \times (m_x - 1)$-tridiagonal with $a_{i,i}=2\kappa + \frac{q}{2}h_x^2$, $a_{i,j}=-\kappa$ for $j=i\pm 1$, 
the matrix $A_{y}$ is $(m_y - 1) \times (m_y - 1)$ defined in a similar way as $A_x$ with $h_x$ replaced by 
$h_y$. Further $I_x$ and $I_y$ are identity matrices of same sizes as $A_x$ and $A_y$ respectively. 

Proceeding in a similar way for \((x_i, y_j, z_k) \in \Omega_h^{(3)}\), we have \eqref{eq:semidiscrete} with 
$$A_h = I_z\otimes I_y \otimes A_x + I_z\otimes A_y \otimes I_x + A_z\otimes I_y \otimes I_x,$$   
where $A_x$ is defined similar to the two-dimension case with the only change on diagonal given by 
$a_{i,i}=2\kappa + \frac{q}{3}h_x^2$. The matrices $A_y$ and $A_z$ are defined in same way and 
$A_z$ is $(m_z - 1) \times (m_z - 1)$. 

Letting a superscript $\top$ denote a matrix/vector transpose, then in two-dimensions, 
$$\mathbf{u}_h(t) = [\mathbf{u}_h^{[1]}(t), \mathbf{u}_h^{[2]}(t), \ldots, \mathbf{u}_h^{[m_y - 1]}(t)]^\top, 
\mathbf{u}_h^{[j]}(t) = [u_{h,1j}(t), u_{h,2j}(t), \ldots, u_{h,(m_x-1)j}(t)]^\top,
$$ 
and 
$$
\mathbf{f}(t) = [\mathbf{f}^{[1]}(t), \mathbf{f}^{[2]}(t), \ldots, \mathbf{f}^{[m_y - 1]}(t)]^\top, 
\mathbf{f}^{[j]}(t) = [f_{1j}(t), f_{2j}(t), \ldots, f_{(m_x-1)j}(t)]^\top. 
$$ 
Similarly, in three-dimensions 
$$\mathbf{u}_h(t) = [\mathbf{u}_h^{[1]}(t), \mathbf{u}_h^{[2]}(t), \ldots, \mathbf{u}_h^{[m_z - 1]}(t)]^\top, 
\mathbf{u}_h^{[k]}(t) = [\mathbf{u}_{h}^{[1][k]}(t), \mathbf{u}_{h}^{[2][k]}(t), \ldots, \mathbf{u}_{h}^{[m_y-1][k]}(t)]^\top,
$$ 
with 
$$
\mathbf{u}_h^{[j][k]}(t) = [u_{h,1jk}(t), u_{h,2jk}(t), \ldots, u_{h,(m_x-1)jk}(t)]^\top,
$$
and the vector $\mathbf{f}(t)$ is defined in a similar way. 

\subsection{Time-stepping Scheme}\label{subsec:time-stepping scheme} 
Consider the temporal mesh points \( t_n = n\tau \), \( n = 0, 1, 2, \ldots, N \), where \( \tau = T/N \). 
The construction of exponential time differencing (ETD) schemes for \eqref{eq:semidiscrete},
\[
\frac{d}{dt}\mathbf{u}_h(t) + A_h \mathbf{u}_h(t) = \mathbf{f}(t, \mathbf{u}_h(t)),
\]
is based on the mild form of the solution, given by
\begin{equation}\label{eq:mild_solution}
	\mathbf{u}_h(t_{n+1}) = e^{-A_h\tau}\mathbf{u}_h(t_n) 
	+ \int_{t_n}^{t_{n+1}} e^{-A_h(t_{n+1}-t)} \mathbf{f}(t, \mathbf{u}_h(t))\, dt, 
	\quad n = 0, 1, \ldots, N-1.
\end{equation}
The desired ETD schemes are then obtained by applying suitable quadrature approximations 
to the integral term. For further details on the derivation and analysis of exponential time 
differencing methods, we refer the reader to \cite{Cox:2002,Hochbruck:2005b} 
and the references therein. 

To extend the ETD approach to multidimensional problems efficiently, we employ a 
dimension-splitting strategy that exploits the separable structure of the spatial operator. 
This construction relies on several fundamental properties of matrices and matrix exponentials, 
which we summarize in the following proposition. 

\begin{proposition}\label{prop:Exp-properties}
	Suppose \( A \) and \( B \) are square matrices with real entries such that \( AB = BA \). Then the following relations hold: 
	\begin{enumerate}[label = \roman*)]
		\item $e^{A+B} = e^{A} e^{B} = e^{B} e^{A}$.
		\item $B$ commutes with every polynomial in $A$, i.e., $p(A) B = B p(A)$ for all polynomials $p$.
		In particular, $e^{A} B = B e^{A}$.
		\item If $B$ is invertible, then $B^{-1} A = A B^{-1}$ and $B^{-1}$ commutes with $e^{A}$,
		that is, $e^{A} B^{-1} = B^{-1} e^{A}$.
	\end{enumerate}
	
\end{proposition}

Now, for the dimension-splitting ETD scheme, we decompose the discrete operator as 
\( A_h = A_1 + A_2 \), where the sub-operators correspond to independent spatial directions. 
For the two-dimensional case, we define
\(
A_1 = I_y \otimes A_x, 
\quad 
A_2 = A_y \otimes I_x,
\)
while for the three-dimensional case we take
\(
A_1 = I_z \otimes I_y \otimes A_x, 
\quad 
A_2 = I_z \otimes A_y \otimes I_x + I_y \otimes I_x \otimes A_z.
\)
By standard properties of the Kronecker product, one verifies that \( A_1 \) and \( A_2 \) commute and that both matrices are invertible in their respective settings. 
Thus, using the first property of Proposition \ref{prop:Exp-properties} we have \(e^{-A_h} = e^{-A_1}e^{-A_2}\). As such, the mild form 
\eqref{eq:mild_solution} can be written as  
\begin{equation}\label{eq:mild_solution_DS}
	\mathbf{u}_h(t_{n+1}) = e^{-\tau A_1} e^{-\tau A_2} \mathbf{u}_h(t_{n}) 
	+ \int_{t_n}^{t_{n+1}} e^{-A_1 (t_{n+1} - t)} 
	\tilde{\mathbf{f}}(t, \mathbf{u}_h(t))\, dt,
\end{equation}
where
\[
\tilde{\mathbf{f}}(t, \mathbf{u}_h(t)) = e^{-A_2 (t_{n+1} - t)} \mathbf{f}(t, \mathbf{u}_h(t)).
\]
Equation~\eqref{eq:mild_solution_DS} serves as the basis for constructing 
dimension-splitting exponential time differencing schemes in two and three spatial dimensions.

To this end, define the dimension split exponential time differencing scheme of interest in this paper: $U_h^n \approx \mathbf{u}_h^n$, \(n = 1, 2, 3, \dots, N-1\), by 
\begin{equation}\label{eq:ETD_linear_integral}
	U_h^{n+1} 
	= e^{-A_1\tau} e^{-A_2\tau} U_h^n 
	+ \int_{t_n}^{t_{n+1}} e^{-A_1(t_{n+1} - t)} \mathbf{p}(t)\, dt, 
\end{equation} 
where $\mathbf{p}(t) \approx \mathbf{f}(t, \mathbf{u}_h(t))$ on each subinterval $\mathcal{I}_n$, $n = 0, 1, 2, \ldots N-1$, is a piecewise linear interpolant:  
$$
\mathbf{p}(t) =  \left[\tilde{\mathbf{f}}^n + \frac{(t-t_n)}{\tau} 
\left(\mathbf{f}(t_{n+1},W_h^{n+1}) - \tilde{\mathbf{f}}^n\right) \right], 
$$ 
with  \(
\tilde{\mathbf{f}}^n = \tilde{\mathbf{f}}(t_n, U_h^n)
\) and 

\begin{align}\label{eq:ETD_const_integral}
	W_h^{n+1} 
	&= e^{-A_1\tau} e^{-A_2\tau} U_h^n 
	+ \int_{t_n}^{t_{n+1}} e^{-A_1(t_{n+1} - t)} \tilde{\mathbf{f}}^n \, dt 
\end{align} 
is introduced to ensure the ETD solution \(U_h^n\) is explicit. 

Evaluating the integrals in \eqref{eq:ETD_linear_integral} and \eqref{eq:ETD_const_integral} in closed form gives the fully discrete formulation:
\begin{equation}\label{eq:ETD_const_fully_discrete}
	W_h^{n+1} 
	= e^{-A_1\tau} e^{-A_2\tau} U_h^n 
	+ A_1^{-1} \left[I - e^{-A_1\tau}\right] e^{-A_2\tau} 
	\mathbf{f}(t_n, U_h^n),
\end{equation}
and
\begin{align}\label{eq:ETD_linear_fully_discrete}
	U_h^{n+1} 
	&= W_h^{n+1} 
	+ \left( A_1^{-1}\tau - A_1^{-2} + A_1^{-2} e^{-A_1\tau} \right)
	\frac{ \mathbf{f}(t_{n+1}, W_h^{n+1}) 
		- e^{-A_2\tau} \mathbf{f}(t_n, U_h^n) }{\tau}.
\end{align}
Equations~\eqref{eq:ETD_const_fully_discrete}--\eqref{eq:ETD_linear_fully_discrete} define 
the two-staged dimension-splitting ETD2-RK scheme, denoted ETD2RK-DS which efficiently exploits 
the separable structure of \( A_h = A_1 + A_2 \).

\section{Numerical Implementation}\label{Sec:Numerical_Aspects} 
This section describes the practical realization of the proposed ETD2RK--DS
scheme. 
We derive explicit, implementation-ready algorithms that realize the
dimension-splitting strategy through systematic tensor reshaping and splicing,
thereby reducing each multidimensional system to collections of
one-dimensional subsystems. 
Efficient solution strategies for these resulting one-dimensional problems are then
developed, followed by complete algorithms for two- and three-dimensional problems.

A central focus of this section is the efficient implementation of the fully
discrete scheme \eqref{eq:ETD_linear_fully_discrete}. To this end, we employ the
Pad\'e approximant $P_{0,2}$ to approximate the matrix exponential and present
two complementary implementation strategies: a Sylvester-equation-based
formulation developed in this work and a classical LU-factorization approach,
which is used as a reference for performance comparison.

We begin by outlining the Pad\'e-based $P_{0,2}$ ETD2RK--DS formulation and its
dimension-split structure, which forms the basis for the implementation
strategies developed in the remainder of this section. 

\subsection{Pad\'e-based ETD2RK--DS formulation}\label{subsec:pade_etd2rk_ds}
For the two-dimension case, replacing the exponential terms in 
\eqref{eq:ETD_const_fully_discrete} and \eqref{eq:ETD_linear_fully_discrete}
by the Pad\'e approximant
\[
P_{0,2}(x) = 2(x^2 + 2x + 2 I)^{-1},
\]
and using the last two properties in Proposition \ref{prop:Exp-properties}, we can have 
\begin{equation}\label{eq:ETD_const_fully_discrete_Rational_2D}
	\widehat{W}_h^{n+1} 
	= Q_1(\tau A_1) Q_1(\tau A_2) \widehat{U}_h^n 
	+ \tau Q_2(\tau A_1) Q_1(\tau A_2) \mathbf{f}(t_n, \widehat{U}_h^n),
\end{equation}
and
\begin{equation}\label{eq:ETD_linear_fully_discrete_Rational_2D}
	\widehat{U}_h^{n+1} 
	= \widehat{W}_h^{n+1} 
	+ \tau Q_3(\tau A_1)
	\left[
	\mathbf{f}(t_{n+1}, \widehat{W}_h^{n+1})
	- Q_1(\tau A_2)\mathbf{f}(t_n, \widehat{U}_h^n)
	\right],
\end{equation}
where the rational functions \( Q_\ell \) are defined as
\begin{align*}
	Q_1(x) &= P_{0,2}(x), 
	& Q_2(x) &= (x + 2I)(x^2 + 2x + 2I)^{-1}, \\
	Q_3(x) &= (x + I)(x^2 + 2x + 2I)^{-1}.
\end{align*}

Moreover, these rational functions admit partial fraction decompositions that 
can be exploited for enhanced computational efficiency. 
Specifically, they can be expressed as
\[
Q_\ell(x) = 2\,\mathrm{Re}\!\left( r_\ell (x - sI)^{-1} \right),
\quad \ell = 1, 2, 3,
\]
where
\(
s = -1 + \iota, \;
r_1 = -\iota, \;
r_2 = \tfrac{1 - \iota}{2}, \;
r_3 = \tfrac{1}{2},
\)
and \( \iota = \sqrt{-1} \).

In the three-dimensional case, we decompose
\(
A_2 = A_2^{(y)} + A_2^{(z)},
\)
where 
\( A_2^{(y)} = I_z \otimes A_y \otimes I_x \) and \( A_2^{(z)} = A_z \otimes I_y \otimes I_x \).
Since these matrices commute, we have
\( e^{-\tau A_2} = e^{-\tau(A_2^{(y)} + A_2^{(z)})} = e^{-\tau A_2^{(y)}} e^{-\tau A_2^{(z)}} \). Then, following the same 
construction, the three-dimensional case yields 
\begin{align}\label{eq:ETD_const_fully_discrete_Rational_3D}
	\widehat{W}_h^{n+1}
	&= Q_1(\tau A_1) 
	Q_1(\tau A_2^{(y)}) 
	Q_1(\tau A_2^{(z)}) \widehat{U}_h^n \nonumber \\
	&\quad
	+ \tau Q_2(\tau A_1) 
	Q_1(\tau A_2^{(y)}) 
	Q_1(\tau A_2^{(z)}) \mathbf{f}(t_n, \widehat{U}_h^n),
\end{align}
and
\begin{equation}\label{eq:ETD_linear_fully_discrete_Rational_3D}
	\widehat{U}_h^{n+1}
	= \widehat{W}_h^{n+1}
	+ \tau Q_3(\tau A_1)
	\left[
	\mathbf{f}(t_{n+1}, \widehat{W}_h^{n+1})
	- Q_1(\tau A_2^{(y)}) Q_1(\tau A_2^{(z)}) 
	\mathbf{f}(t_n, \widehat{U}_h^n)
	\right].
\end{equation}
We refer to \eqref{eq:ETD_const_fully_discrete_Rational_2D}-\eqref{eq:ETD_linear_fully_discrete_Rational_2D} 
and \eqref{eq:ETD_const_fully_discrete_Rational_3D}-\eqref{eq:ETD_linear_fully_discrete_Rational_3D} as the Pad\'e 
based ETD2RK-DS in the two and three dimensional cases, respectively. 

Since implementing the above schemes requires solving linear systems in place of the multiplication of matrix exponential with vectors,  
we next show how dimension splitting can be employed to decompose these systems into a collection of smaller, one-dimensional problems. 

\subsection{Dimension Splitting and Problem-Size Reduction}\label{subsec:Dimension_Split}
We now discuss how the dimension splitting is employed to decompose the multidimensional problem 
into a collection of smaller, one-dimensional systems. 
For clarity, we first present the two-dimensional case. 
Recall that \( A_1 = I_y \otimes A_x \) and \( A_2 = A_y \otimes I_x \).

Consider evaluating the matrix-vector product 
\( Q_\ell(\tau A_\nu)\mathbf{g} \), where 
\( \mathbf{g} \in \mathbb{R}^{(m_x-1)(m_y-1)} \) 
is a vector corresponding to the two-dimensional discretization grid,
\[
\mathbf{g} = [\mathbf{g}^{[1]}, \mathbf{g}^{[2]}, \ldots, \mathbf{g}^{[m_y - 1]}]^{\mathsf{T}}, 
\quad 
\mathbf{g}^{[j]} = [g_{1j}, g_{2j}, \ldots, g_{(m_x-1)j}]^{\mathsf{T}}.
\]
Recalling the partial-fraction representation of \( Q_\ell \), we can write
\[
Q_\ell(\tau A_\nu)\mathbf{g} 
= 2\,\mathrm{Re}\!\left( r_\ell (\tau A_\nu - sI)^{-1} \mathbf{g} \right),
\qquad 
\nu = 1, 2, \quad \ell = 1, 2, 3,
\]
which is equivalent to solving the linear system
\begin{equation}\label{eq:linear system xyz direction}
	(\tau A_\nu - sI)\mathbf{v} = r_\ell \mathbf{g},
\end{equation}
so that \( Q_\ell(\tau A_\nu)\mathbf{g} = 2\,\mathrm{Re}(\mathbf{v}) \). 
Here, \( \mathbf{v} \in \mathbb{R}^{(m_x-1)(m_y-1)} \) is defined analogously to \( \mathbf{g} \).

\subsubsection{Two-Dimensional Case}\label{subsec:2D Dimension Spliting}
Starting with \( A_1 \), note that \( A_1 = I_y \otimes A_x \) is a block tridiagonal matrix 
whose diagonal blocks are all identical and equal to \( A_x \). 
Hence, the linear system \eqref{eq:linear system xyz direction} with \( \nu = 1 \) 
decouples into \( m_y - 1 \) independent subsystems, 
each with coefficient matrix \( A_x \). 
Therefore, solving \eqref{eq:linear system xyz direction} for \( \nu = 1 \) 
reduces to solving the one-dimensional problems
\begin{equation}\label{eq:linear system x direction 1D}
	(\tau A_x - sI)\mathbf{v}^{[j]} = r_\ell \mathbf{g}^{[j]},
	\qquad j = 1, 2, \ldots, m_y - 1.
\end{equation}

For \( \nu = 2 \), unlike the case of \( A_1 \), the structure of \( A_2 = A_y \otimes I_x \) 
does not immediately yield decoupled blocks. 
However, by carefully permuting the vectorization order of the unknowns, 
the system can be rearranged into a form that separates the \( y \)-direction. 
Specifically, solving 
\(
(\tau A_2 - sI)\mathbf{v} = r_\ell \mathbf{g}
\)
is equivalent to solving
\(
(\tau \tilde{A}_2 - sI)\tilde{\mathbf{v}} = \tilde{\mathbf{g}},
\)
followed by reshaping \(\tilde{\mathbf{v}}\) to recover \(\mathbf{v}\), where
\[
\tilde{\mathbf{g}} 
= [\tilde{\mathbf{g}}^{[1]}, \tilde{\mathbf{g}}^{[2]}, \ldots, \tilde{\mathbf{g}}^{[m_x - 1]}]^{\mathsf{T}},
\quad 
\tilde{\mathbf{g}}^{[i]} 
= [g_{i1}, g_{i2}, \ldots, g_{i(m_y-1)}]^{\mathsf{T}}, 
\]
and \(\tilde{\mathbf{v}}\) is defined the same way.
The corresponding coefficient matrix is
\(
\tilde{A}_2 = I_x \otimes A_y,
\)
which consists of \( m_x - 1 \) decoupled one-dimensional blocks, 
each with coefficient matrix \( A_y \). 
Consequently, solving \eqref{eq:linear system xyz direction} for \( \nu = 2 \) 
amounts to computing
\begin{equation}\label{eq:linear system y direction 1D}
	(\tau A_y - sI)\tilde{\mathbf{v}}^{[i]} = r_\ell \tilde{\mathbf{g}}^{[i]},
	\qquad i = 1, 2, \ldots, m_x - 1,
\end{equation}
and subsequently recovering \( \mathbf{v} \) by appropriately reshuffling 
the entries of \( \tilde{\mathbf{v}} \).

\subsubsection{Three-Dimension Case} 
Similar to the two-dimensional case, the matrix 
\( A_1 = I_z \otimes I_y \otimes A_x \) 
is block tridiagonal, with \( A_x \) repeated along its diagonal blocks. 
Consequently, solving a linear system with \( A_1 \) as its coefficient matrix 
reduces to solving \( (m_y - 1)\cdot(m_z - 1) \) one-dimensional systems, 
each with coefficient matrix \( A_x \).  

In particular, consider the linear system \eqref{eq:linear system xyz direction} 
with \( \nu = 1 \) and 
\[
\mathbf{g} 
= [\mathbf{g}^{[1]}, \mathbf{g}^{[2]}, \ldots, \mathbf{g}^{[m_z - 1]}]^{\mathsf{T}},
\qquad 
\mathbf{g}^{[k]} 
= [\mathbf{g}^{[1][k]}, \mathbf{g}^{[2][k]}, \ldots, \mathbf{g}^{[m_y - 1][k]}]^{\mathsf{T}},
\]
where each subvector 
\(
\mathbf{g}^{[j][k]} = [g_{1jk}, g_{2jk}, \ldots, g_{(m_x - 1)jk}]^{\mathsf{T}}.
\)
Defining \( \mathbf{v} \) in an analogous manner, 
the problem reduces to the set of one-dimensional systems
\begin{equation}\label{eq:linear system x direction 1D 3D case}
	(\tau A_x - sI_x)\mathbf{v}^{[j][k]} = r_\ell \mathbf{g}^{[j][k]},
	\quad 
	j = 1, 2, \ldots, m_y - 1, \text{ for each } 
	\,\,k = 1, 2, \ldots, m_z - 1.
\end{equation}

For \( \nu = 2 \) in \eqref{eq:linear system xyz direction}, 
the structure of \( A_2 = A_2^{(y)} + A_2^{(z)} \) 
is not block diagonal in its native form. 
However, as in the two-dimensional case, 
block diagonalization can be achieved by a suitable rearrangement 
of the vectorized system. 
Starting with \( A_2^{(y)} \), 
solving \eqref{eq:linear system xyz direction} with 
\( A_\nu = A_2^{(y)} \) 
is equivalent to solving
\(
(\tau \tilde{A}_2^{(y)} - sI)\tilde{\mathbf{v}} = \tilde{\mathbf{g}},
\)
where \( \tilde{A}_2^{(y)} = I_x \otimes I_z \otimes A_y \) and
\[
\mathbf{g} 
= [\mathbf{g}^{[1]}, \mathbf{g}^{[2]}, \ldots, \mathbf{g}^{[m_z - 1]}]^{\mathsf{T}}, 
\qquad
\tilde{\mathbf{g}}^{[k]} 
= [\tilde{\mathbf{g}}^{[1][k]}, \tilde{\mathbf{g}}^{[2][k]}, \ldots, \tilde{\mathbf{g}}^{[m_x - 1][k]}]^\top,
\]
with each
\(
\tilde{\mathbf{g}}^{[i][k]} = [g_{i1k}, g_{i2k}, \ldots, g_{i(m_y - 1)k}]^{\mathsf{T}}.
\)
Under this reformulation, the vector \( \tilde{\mathbf{v}} \) 
is obtained by solving the set of one-dimensional systems
\begin{equation}\label{eq:linear system y direction 1D 3D case}
	(\tau A_y - sI_y)\tilde{\mathbf{v}}^{[i][k]} = r_\ell \tilde{\mathbf{g}}^{[i][k]},
	\quad 
	i = 1, 2, \ldots, m_x - 1, \text{ for each } 
	k = 1, 2, \ldots, m_z - 1.
\end{equation}

Similarly, for \( A_\nu = A_2^{(z)} \) in 
\eqref{eq:linear system xyz direction}, 
we consider the reformulated system
\(
(\tau \check{A}_2^{(z)} - sI)\check{\mathbf{v}} = \check{\mathbf{g}},
\)
where \( \check{A}_2^{(z)} = I_x \otimes I_y \otimes A_z \), and
\[
\check{\mathbf{g}} 
= [\check{\mathbf{g}}^{[1]}, \check{\mathbf{g}}^{[2]}, \ldots, \check{\mathbf{g}}^{[m_y - 1]}]^{\mathsf{T}}, 
\qquad
\check{\mathbf{g}}^{[j]} 
= [\check{\mathbf{g}}^{[1][j]}, \check{\mathbf{g}}^{[2][j]}, \ldots, 
\check{\mathbf{g}}^{[m_x - 1][j]}]^{\mathsf{T}},
\]
with
\(
\check{\mathbf{g}}^{[i][j]} 
= [g_{ij1}, g_{ij2}, \ldots, g_{ij(m_z - 1)}]^{\mathsf{T}}.
\)
In this case, the vector \( \check{\mathbf{v}} \) 
is obtained by solving \( (m_x - 1)(m_y - 1) \) 
independent one-dimensional problems of size \( (m_z - 1) \):
\begin{equation}\label{eq:linear system z direction 1D 3D case}
	(\tau A_z - sI_z)\check{\mathbf{v}}^{[i][j]} = r_\ell \check{\mathbf{g}}^{[i][j]},
	\quad 
	i = 1, 2, \ldots, m_x - 1, \text{ for each } 
	j = 1, 2, \ldots, m_y - 1.
\end{equation} 

The slicing and rearrangement procedures described above transform a single multidimensional linear solve into a collection of independent 
one-dimensional problems. For a 2D grid with $m$ points per direction, we solve $\mathcal{O}(m)$ tridiagonal systems of size $\mathcal{O}(m)$, 
yielding total cost $\mathcal{O}(m^2)$. This represents an order-of-magnitude improvement over solving the original $m^2 \times m^2$ banded 
system, which would cost $\mathcal{O}(m^3)$ operations per solve. In 3D, the gain is even more substantial: from $\mathcal{O}(m^5)$ to $\mathcal{O}(m^3)$.	

\subsection{Efficient Linear Solvers}
We now discuss efficient numerical techniques for solving the one-dimensional systems derived in the previous subsection. 
Two approaches are considered: a classical LU factorization method and a Sylvester-equation-based formulation 
developed here for the ETD schemes. The LU-based method is presented first, primarily to serve as a baseline for comparison. 

\subsubsection{LU Factorization-Based Implementation}
In any of the one-dimensional systems such as \eqref{eq:linear system x direction 1D} or 
\eqref{eq:linear system y direction 1D}, the coefficient matrix does not depend on the index \( i \). 
Hence, a single LU factorization can be computed once and reused for all subsequent right-hand sides. 
Specifically, we compute
\[
L_x U_x = \tau A_x - sI_x, \quad \text{ so that } 
L_x U_x \mathbf{v}^{[i]} = r_\ell \mathbf{g}^{[i]}.
\]
The solution proceeds via the standard forward and backward substitution steps:
\[
L_x \boldsymbol{\theta}_x^{[i]} = r_\ell \mathbf{g}^{[i]}, 
\qquad 
U_x \mathbf{v}^{[i]} = \boldsymbol{\theta}_x^{[i]}.
\]
The same procedure applies for systems in the \( y \)- and \( z \)-directions, 
using the corresponding matrices \( A_y \) and \( A_z \). 
Once the LU factors are available, the computational cost per right-hand side is minimal, 
making this approach suitable for problems with a moderate number of grid points in each spatial direction.

\subsubsection{Sylvester Equation-Based Implementation}
We next consider an alternative formulation based on the Sylvester equation, 
which offers a more efficient solution strategy for systems with diagonalizable 
operators. Consider again the one-dimensional problem \eqref{eq:linear system x direction 1D}, 
which can be written equivalently as
\begin{equation}\label{eq:Sylvester_matrix_vector}
	A_x \mathbf{v}^{[i]} - s I_x \mathbf{v}^{[i]} = r_\ell \mathbf{g}^{[i]}.
\end{equation}
This is a special case of the Sylvester equation \( AX + XB = C \), 
for which several efficient algorithms have been documented; see, e.g., 
\cite{Casulli:2024,Teran:2011}. 
An efficient diagonalization-based variant, discussed in 
\cite{Garrappa:2022,Sarumi:2025}, 
is adopted in this work.

Since \( A_x \) is symmetric and has distinct eigenvalues, it is diagonalizable as
\(
A_x = P_x D_x P_x^{\mathsf{T}},
\)
where \( P_x \) is orthogonal. 
Applying this transformation to \eqref{eq:Sylvester_matrix_vector} gives
\[
(D_x - sI_x)\,\bar{\mathbf{v}}^{[i]} = r_\ell \bar{\mathbf{g}}^{[i]},
\]
where 
\( \bar{\mathbf{v}}^{[i]} = P_x^{\mathsf{T}}\mathbf{v}^{[i]} \) 
and 
\( \bar{\mathbf{g}}^{[i]} = P_x^{\mathsf{T}}\mathbf{g}^{[i]} \).
The solution can be obtained efficiently via componentwise (Hadamard) division:
\[
\bar{\mathbf{v}}^{[i]} = r_\ell\, \bar{\mathbf{g}}^{[i]} \oslash \mathbf{d},
\]
where \( \mathbf{d} = \mathrm{diag}(D_x - sI_x) \) 
is the column vector containing the diagonal entries of \( D_x - sI_x \).
Finally, the solution in the original coordinate system is recovered as
\[
\mathbf{v}^{[i]} = P_x \bar{\mathbf{v}}^{[i]}.
\]
This diagonalization-based solver requires the eigen-decomposition of \( A_x \) only once, 
after which all subsequent systems can be solved with negligible cost 
using matrix-vector multiplications and elementwise operations. 
Such reuse makes the method particularly well suited for large-scale problems 
where multiple right-hand sides arise at each time step.	

\subsection{Algorithmic Realization} 
We now outline a sequence of computational steps for implementing the proposed ETD2RK-DS scheme. 
Let 
\(
\mathbf{w}_1 = Q_1(\tau A_2)\widehat{U}_h^n
\)
and 
\(
\mathbf{w}_2 = Q_1(\tau A_2)\mathbf{f}^n,
\)
where 
\(
\mathbf{f}^n = \mathbf{f}(t_n, \widehat{U}_h^n).
\)
Subsequently, we define
\(
\mathbf{w}_3 = Q_1(\tau A_2)\mathbf{w}_1
\)
and 
\(
\mathbf{w}_4 = Q_2(\tau A_2)\mathbf{w}_2.
\)

Since the matrices \( A_x \), \( A_y \), and \( A_z \) arise from spatial discretizations that are independent of both the time level \( t_n \) and spatial index \( i \), their spectral decompositions,
\(\tau A_\nu = P_\nu D_\nu P_\nu^{\mathsf{T}}, \,\, \nu \in \{x, y, z\},\)
can be computed once at the beginning of the simulation and reused throughout the entire time integration.

Following the discussion in Subsection~\ref{subsec:2D Dimension Spliting}, 
the expressions for \( \mathbf{w}_1 \) and \( \mathbf{w}_2 \) are reformulated in terms of one-dimensional problems. 
For instance, in the two-dimensional case, this reduces to solving, for 
\( i = 1, 2, \ldots, m_x - 1 \),
\[
(\tau A_y - sI_y)\,\widetilde{\mathbf{w}}_1^{n[i]} = r_\ell\,\widetilde{\widehat{U}}^{n[i]},
\qquad 
(\tau A_y - sI_y)\,\widetilde{\mathbf{w}}_2^{n[i]} = r_\ell\,\widetilde{\mathbf{f}}^{n[i]}.
\]
These systems are then solved efficiently using the precomputed diagonalizations, 
as detailed in Algorithms~\ref{algo:2D Implementation} and~\ref{algo:3D Implementation}. 
Throughout the algorithms, the symbol $\odot$ denotes componentwise (Hadamard) multiplication of vectors.

\begin{algorithm} 
	\caption{Implementation of the ETD2RK-DS for two-dimension problem}\label{algo:2D Implementation}
	\begin{algorithmic}
		\State Step 1: Precompute $\{P_x, D_x, P_x^\top\}$ for $A_x$,  and  $\{P_y, D_y, P_y^\top\}$ for $A_y$
		
		\State Step 2: Precompute column vectors $\mathbf{d}_x^{(\ell)} = \text{Re}(r_\ell \oslash \text{diag}(D_x - sI_x))$, \(\ell = 1, 2, 3, \,\, \)  
		and $\mathbf{d}_y^{(1)} = \text{Re}(r_1\oslash \text{diag}(\tau D_y - sI_y))$ 
		
		\State Step 3: given $U_h^0$ and $\mathbf{f}$, then for each time mesh $t_n, \,\, n\in {0, 1, \dots, N-1}$  
		
		\Statex \hspace{1em}Step 3.1: (\textbf{Compute} \(\mathbf{w}_1 = Q_1(\tau A_2) \widehat{U}_h^n\) and \(\mathbf{w}_2  = Q_1(\tau A_2) \mathbf{f}(t_n, \widehat{U}_h^n)\)) \Statex \hspace{2em}for $i = 1, 2, \dots, m_x - 1$  
		\Statex \hspace{2em}3.1.1: $\widebar{\widetilde{\widehat{U}}}^{n[i]} = P_y^\top\widetilde{\widehat{U}}^{n[i]} $,  $\widebar{\tilde{\mathbf{f}}}^{n[i]} = P_y^\top \tilde{\mathbf{f}}^{n[i]} $ 
		\Statex \hspace{2em}3.1.2: $ \widetilde{\mathbf{w}} _1^{n[i]} = 2P_y\cdot \left(\widebar{\widetilde{\widehat{U}}}^{n[i]} \odot \mathbf{d}_y^{(1)}\right)$  
		\Statex \hspace{2em}3.1.3: $ \widetilde{\mathbf{w}}_2^{n[i]} = 2P_y \cdot \left(\widebar{\tilde{\mathbf{f}}}^{n[i]} \odot \mathbf{d}_y^{(1)}\right)$   
		
		\State \hspace{1em}Step 3.2: for $\nu = 1, 2$ 
		\Statex \hspace{2em}3.2.1: concatenate the vectors $\widetilde{\mathbf{w}}_\nu^{n[i]}$, \(i = 1, 2, \dots, m_x - 1,\) to construct $\widetilde{\mathbf{w}}_\nu$  
		\Statex \hspace{2em}3.2.2: rearrange the entries of $\widetilde{\mathbf{w}}_\nu$ to obtain $\mathbf{w}_\nu$  
		\State \hspace{1em}Step 3.3: (\textbf{Compute} \(\widehat{W}^{n+1}_h\) and \(\widehat{U}_h^{n+1}\)) 
		\Statex \hspace{2em}for $j = 1, 2, \dots, m_y - 1$  
		\Statex \hspace{2em}3.3.1: $\widebar{\mathbf{w}}_1^{n[j]} = P_x^\top \mathbf{w}_1^{n[j]}$, $\quad$ $\widebar{\mathbf{w}}_2^{n[j]} = P_x^\top \mathbf{w}_2^{n[j]}$ 
		\Statex \hspace{2em}3.3.2: $ \widebar{\mathbf{w}}_3^{n[j]} = \widebar{\mathbf{w}}_1^{n[j]} \odot \mathbf{d}_x^{(1)}$,  $\quad$ $ \widebar{\mathbf{w}}_4^{n[j]} = \widebar{\mathbf{w}}_2^{n[j]} \odot \mathbf{d}_x^{(2)}$
		\Statex \hspace{2em}3.3.3: $\widehat{W}_h^{(n+1)[j]} = 2P_x \cdot (\widebar{\mathbf{w}}_3^{n[j]} + \tau \widebar{\mathbf{w}}_4^{n[j]})$  
		\Statex \hspace{2em}3.3.4: $\widebar{\mathbf{g}}^{n[j]} = P_x^\top \mathbf{f}(t_{n+1}, \widehat{W}_h^{(n+1)[j]})  - \widebar{\mathbf{w}}_2^{n[j]}$ 
		\Statex \hspace{2em}3.3.5: $\mathbf{w}_5^{n[j]} = 2P_x\cdot (\widebar{\mathbf{g}}^{n[j]} \odot \mathbf{d}_x^{(3)})$ 
		\Statex \hspace{2em}3.3.6: $\widehat{U}_h^{(n+1)[j]} = \widehat{W}_h^{(n+1)[j]} + \mathbf{w}_5^{n[j]}$  
		
		\State \hspace{1em}Step 3.4: concatenate the vectors $\widehat{U}_h^{(n+1)[j]}$, \(j = 1, 2, \dots, m_y - 1,\) to construct $\widehat{U}_h^{n+1}$  
	\end{algorithmic}
\end{algorithm} 

\begin{algorithm}
	\caption{Implementation of the ETD2RK-DS for the three-dimensional problem}\label{algo:3D Implementation} 
	\begin{algorithmic}
		\State Step 1: Precompute $\{P_z, D_z, P_z^\top\}$ for $A_z$, $\{P_y, D_y, P_y^\top\}$ for $A_y$, and $\{P_x, D_x, P_x^\top\}$ for $A_x$ 
		
		\State Step 2: Precompute the column vectors $\mathbf{d}_x^{(\ell)} = \mathrm{Re}(r_\ell \oslash \mathrm{diag}(\tau D_x - sI))$, \(\ell = 1, 2, 3\),  
		$\mathbf{d}_y^{(1)} = \mathrm{Re}(r_1 \oslash \mathrm{diag}(\tau D_y - sI))$, and $\mathbf{d}_z^{(1)} = \mathrm{Re}(r_1 \oslash \mathrm{diag}(\tau D_z - sI))$
		
		\State Step 3: given $U^0$ and $\mathbf{f}$, then for each index $n\in {0, 1, \dots, N-1}$  
		\Statex \hspace{1em}Step 3.1: [z--direction, solve: \(\mathbf{w}_1 = Q_1(\tau A_2^{(z)})\widehat{U}_h^n\) and \(\mathbf{w}_2 = Q_1(\tau A_2^{(z)}) \mathbf{f}(t_n, \widehat{U}_h^n)\)] \\
		\Statex \hspace{1em}For $j = 1,\dots,m_y-1$; for $i = 1,\dots,m_x-1$:
		\Statex \hspace{2em}3.1.1 $\widebar{\widecheck{\widehat{U}}}^{n[i][j]} = P_z^\top\widecheck{\widehat{U}}^{n[i][j]}, \quad \widebar{\widecheck{\mathbf{f}}}^{n[i][j]} = P_z^\top \widecheck{\mathbf{f}}^{n[i][j]} $  
		\Statex \hspace{2em}3.1.2 $ \widecheck{\mathbf{w}} _1^{n[i][j]} = 2P_z \cdot \bigg(\widebar{\widecheck{\widehat{U}}}^{n[i][j]} \odot \mathbf{d}_z^{(1)} \bigg)$
		\Statex \hspace{2em}3.1.3 $ \widecheck{\mathbf{w}}_2^{n[i][j]} =  2P_z  \cdot \bigg(\widebar{\widecheck{\mathbf{f}}}^{n[i][j]} \odot \mathbf{d}_z^{(1)} \bigg)$   
		\Statex \hspace{1em}Step 3.2: for $\nu = 1, 2$ 
		\Statex \hspace{2em}3.2.1 concatenate the vectors $\widecheck{\mathbf{w}}_\nu^{n[i][j]}$ to construct $\widecheck{\mathbf{w}}_\nu$  
		\Statex \hspace{2em}3.2.2 rearrange the entries of $\widecheck{\mathbf{w}}_\nu$ to obtain $\mathbf{w}_\nu$  
		\Statex \hspace{1em}Step 3.3: [y--direction, solve: \(\mathbf{w}_3 = Q_1(\tau A_2^{(y)})\mathbf{w}_1\) and \(\mathbf{w}_4 = Q_1(\tau A_2^{(y)})\mathbf{w}_2\)] 
		\\
		\Statex \hspace{1em}For $k = 1, 2, \dots, m_z - 1$; for $i = 1, 2, \dots, m_x - 1$:
		\Statex \hspace{2em}3.3.1 $\widebar{\widetilde{\mathbf{w}}}_1^{n[i][k]} = P_y^\top\widetilde{\mathbf{w}}_1^{n[i][k]}, \quad \widebar{\widetilde{\mathbf{w}}}_2^{n[i][k]} = P_y^\top\widetilde{\mathbf{w}}_2^{n[i][k]} $ 
		\Statex \hspace{2em}3.3.2 $ \widetilde{\mathbf{w}}_3^{n[i][k]} = 2P_y \cdot \bigg(\widebar{\widetilde{\mathbf{w}}}_1^{n[i][k]} \odot \mathbf{d}_y^{(1)}\bigg)$  
		\Statex \hspace{2em}3.3.3 $ \widetilde{\mathbf{w}}_4^{n[i][k]} = 2P_y \cdot\bigg(\widebar{\widetilde{\mathbf{w}}}_2^{n[i][k]}\odot \mathbf{d}_y^{(1)}\bigg)$   
		\Statex \hspace{1em}Step 3.4: for $\nu = 3, 4$ 
		\Statex \hspace{2em}3.4.1 concatenate the vectors $\widetilde{\mathbf{w}}_\nu^{n[i][k]}$ to construct $\widetilde{\mathbf{w}}_\nu$  
		\Statex \hspace{2em}3.4.2 rearrange the entries of $\widetilde{\mathbf{w}}_\nu$ to obtain $\mathbf{w}_\nu$ 
		\Statex \hspace{1em}Step 3.5: [x--direction solve and ETD update: Compute \(\widehat{W}^{n+1}_h\) and \(\widehat{U}^{n+1}_h\)]
		\\
		\Statex \hspace{1em}For $k = 1, 2, \dots, m_z - 1$; for $j = 1, 2, \dots, m_y - 1$ 
		\Statex \hspace{2em}3.5.1 $\widebar{\mathbf{w}}_3^{n[j][k]} = P_x^\top \mathbf{w}_3^{n[j][k]}, \quad \widebar{\mathbf{w}}_4^{n[j][k]} = P_x^\top \mathbf{w}_4^{n[j][k]}$ 
		\Statex \hspace{2em}3.5.2 $ \widebar{\mathbf{w}}_5^{n[j][k]} = \widebar{\mathbf{w}}_3^{n[j][k]} \odot \mathbf{d}_x^{(1)}$  
		\Statex \hspace{2em}3.5.3 $ \widebar{\mathbf{w}}_6^{n[j][k]} = \widebar{\mathbf{w}}_4^{n[j][k]} \odot \mathbf{d}_x^{(2)}$   
		\Statex \hspace{2em}3.5.4 $\widehat{W}_h^{(n+1)[j][k]} = 2P_x\cdot(\widebar{\mathbf{w}}_5^{n[j][k]} + \tau \widebar{\mathbf{w}}_6^{n[j][k]})$ 
		\Statex \hspace{2em}3.5.5 $\widebar{\mathbf{g}}^{n[j][k]} = P_x^\top \mathbf{f}(t_{n+1}, \widehat{W}_h^{(n+1)[j][k]})  - \widebar{\mathbf{w}}_4^{n[j][k]}$ 
		\Statex \hspace{2em}3.5.6 $\mathbf{w}_7^{n[j][k]} = 2P_x (\widebar{\mathbf{g}}^{n[j][k]} \odot \mathbf{d}_x^{(3)})$ 
		\Statex \hspace{2em}3.5.7 $\widehat{U}_h^{(n+1)^{[j][k]}} = \widehat{W}_h^{(n+1)^{[j][k]}} + \mathbf{w}^{n[j][k]}_7$  
		
		\Statex \hspace{1em}Step 3.6: concatenate the vectors $\widehat{U}_h^{(n+1)^{[j][k]}}$, [$j = 1, 2, \dots, m_y - 1$, for each $k = 1, 2, \dots, m_z - 1$] to construct $\widehat{U}_h^{n+1}$ 
	\end{algorithmic}
\end{algorithm}

\begin{remark}\label{rem:algorithms}
	\begin{enumerate}			
		\item The complex-valued calculations in Algorithms \ref{algo:2D Implementation} and \ref{algo:3D Implementation} are confined to the Hadamard division precomputed in Step~2; 
		thereafter, the time stepping can be carried out using real-valued matrix–vector operations.
		
		\item The nested loops appearing in the algorithms can be bypassed by formulating the operations 
		in matrix-matrix form rather than matrix-vector form. 
		This reformulation enables the use of high-performance routines 
		and can significantly improve computational efficiency on modern hardware. 
		
		\item The Sylvester equation type reformulation (diagonalization-based) solvers used in Algorithms~\ref{algo:2D Implementation} 
		and~\ref{algo:3D Implementation} can be replaced by classical LU-based solvers with only minor 
		modifications to the overall algorithmic structure. 
	\end{enumerate}
\end{remark}

\section{Stability and Convergence Analysis}\label{sec:analysis} 
To facilitate the subsequent stability and convergence analysis, 
we introduce the following notation and standing assumptions.

We adopt the submultiplicative infinity norm, defined for a vector 
$\mathbf{b} \in \mathbb{R}^m$ and a matrix $B \in \mathbb{R}^{m\times m}$ by
\(\|\mathbf{b}\| = \max_{1 \le i \le m} |b_i|\), and 
\(\|B\| = \max_{1 \le j \le m} \sum_{i=1}^m |b_{ij}|\), 
respectively. 
The initial data $u_0(\mathbf{x})$ is assumed to be uniformly bounded so that
\(\|U_h^0\| \le \sup_{\mathbf{x} \in \Omega}|u_0(\mathbf{x})| =: M_0,\)
where $M_0$ is independent of the mesh parameters.

\begin{assumption}[Local boundedness and local Lipschitz on bounded sets]\label{Assump:boundedness_Lipschitz}
	\begin{enumerate}
		\item 
		\textbf{Locally Lipschitz source} Let $\mathcal{D} \subset \mathbb{R}^d$ be a suitable region. We assume that the function 
		$\mathbf{f}: (0, T] \times \mathcal{D} \to \mathbb{R}^d$ is locally Lipschitz continuous in the second 
		variable, that is there exists a constant $L = L(T,R) > 0$ such that 
		\begin{equation*}
			||\mathbf{f}(t,\mathbf{u}_h(t)) - \mathbf{f}(t,\tilde{\mathbf{u}}_h(t))|| \le L||\mathbf{u}_h(t) - \tilde{\mathbf{u}}_h(t)||, 
		\end{equation*} 
		for all $t \in (0, T]$ and $\max(\|\mathbf{u}_h(t)\|, \|\tilde{\mathbf{u}}_h(t)\|) \le R$.
		
		\item \textbf{Boundedness of $f$ on bounded sets.} 
		Fix $T>0$. There exist real number $C_b = C_b(M)$ such that
		\[
		\|u\|\le M \Rightarrow \sup_{t\in(0,T]}\|f(t,u)\| \le C_b(M).
		\] 
		Equivalently: $\mathbf{f}$ is bounded whenever $\mathbf{u}_h$ is bounded (the bound depends on $M$).
	\end{enumerate}
\end{assumption} 
In addition, the matrices $A_1$ and $A_2$ arising from the spatial discretization 
are diagonally dominant. 
Hence, following \cite{Du:2019}, the associated exponential propagators satisfy
\begin{equation}\label{ineq:exp_estimate}
	\|e^{-\tau A_i}\| \le e^{-\frac{q h^2}{2}} \le 1, 
	\qquad i=1,2 
	\quad \text{(two-dimensional case)},
\end{equation}
and in three dimensions a similar estimate holds,
\(\|e^{-\tau A_i}\| \le e^{-\frac{q h^2}{3}}.\)

\begin{theorem}[Stability]\label{thm:stability_ETD_dimension_split}
	Let $U_h^n$ be the fully discrete ETD solution produced by the dimension-splitting scheme \eqref{eq:ETD_linear_fully_discrete} with timestep 
	$\tau>0$ and $t_n=n\tau\le T$. Assume the source function \(f\) satisfies part 2 of assumption \ref{Assump:boundedness_Lipschitz} and the 
	bound \(C_b\) satisfies a linear growth condition, that is, there exists constants \(c_1, c_2 > 0\) such that \(C_b(M) \le c_1 M + c_2\).  
	Then there exists a constant $M_*>0$, depending only on $(M_0,T,c_1,c_2)$ (but \emph{independent} of $h,\tau,n$ and \(N\)), such that
	\[
	\|U_h^n\|\;\le\; M_* \qquad \text{for all } n \text{ with } t_n\le T.
	\]
	In particular, the ETD2RK-DS scheme is (uniformly) stable on $[0,T]$.
\end{theorem}  

\begin{proof}
	First, we derive an estimate on the helper scheme \(W^{n+1}_h\) using the integral form \eqref{eq:ETD_const_integral}:
	\[
	W_h^{n+1}
	= e^{-\tau A_1} e^{-\tau A_2} U_h^n
	+ \int_{t_n}^{t_{n+1}} e^{-A_1 (t_{n+1}-s)}\, \tilde{\mathbf f}^n \, ds,
	\quad
	\tilde{\mathbf f}^n := e^{-\tau A_2}\, \mathbf{f}(t_n,U_h^n).
	\]
	Applying triangle inequality, the submultiplicativity property of the norm, and the estimate \eqref{ineq:exp_estimate}: 
	\begin{equation}\label{ineq:Wn_pre_stab}
		\|W_h^{n+1}\|
		\le \|U_h^n\| + \tau\,\|\tilde{\mathbf f}^n\|
		\le \|U_h^n\| + \tau\,\|\mathbf{f}(t_n,U_h^n)\|.
	\end{equation}
	Similarly, for the second substep (linear interpolation of the source):
	\[
	U_h^{n+1}
	= e^{-\tau A_1} e^{-\tau A_2} U_h^n
	+ \int_{t_n}^{t_{n+1}} e^{-A_1 (t_{n+1}-s)}
	\Big[ \tilde{\mathbf f}^n + \tfrac{s-t_n}{\tau}\big( f(t_{n+1},W_h^{n+1}) - \tilde{\mathbf f}^n \big)\Big] ds, 
	\]
	applying the triangle inequality, the submultiplicativity property of the norm, and the estimate \eqref{ineq:exp_estimate} again \eqref{ineq:Wn_pre_stab}, we get that  
	\begin{align}\label{ineq:Un_pre_stab}
		\|U_h^{n+1}\|
		&\le \|U_h^n\|
		+ \int_{t_n}^{t_{n+1}}\Big(\|\tilde{\mathbf f}^n\|
		+ \tfrac{s-t_n}{\tau}\big(\|\mathbf{f}(t_{n+1},W_h^{n+1})\|+\|\tilde{\mathbf f}^n\|\big)\Big)ds 
		\nonumber \\
		&\le \|U_h^n\| + \tau\,\|\tilde{\mathbf f}^n\| + \frac{\tau}{2}\big(\|\mathbf{f}(t_{n+1},W_h^{n+1})\|+\|\tilde{\mathbf f}^n\|\big) 
		\nonumber \\ 
		&\le \|U_h^n\| + \tau \|\mathbf{f}(t_n, U^n_h)\| + \frac{\tau}{2}\big(\|\mathbf{f}(t_{n+1},W_h^{n+1})\|+\|\mathbf{f}(t_n, U^n_h)\|\big)
	\end{align}
	We shall now proceed by induction to show that each \(U_h^n\) satisfies 
	\begin{equation}\label{ineq:Un_stab}
		\|U^n_h\| \le M_{n-1} + \frac{3}{2} \tau C_b(M_{n-1}) + \frac{\tau}{2} C_b(M_{n-1} + \tau C_b(M_{n-1})) =: M_n, 
	\end{equation} 
	for \(n = 1, 2, \dots, N\).
	
	Base case: \(U_h^1\), \(n = 0\) in \eqref{ineq:Un_pre_stab}, we have 
	\begin{align*}
		\|U_h^{1}\| 
		&\le \|U_h^0\| + \tau\,\|\mathbf{f}(t_0, U^0_h)\| + \frac{\tau}{2}\big(\|\mathbf{f}(t_{1},W_h^{1})\|+\|\mathbf{f}(t_0, U^0_h)\|\big).
	\end{align*}
	First, an estimate on \(\|W_h^1\|\) must be determined in order to use the local boundedness assumption in \ref{Assump:boundedness_Lipschitz}. 
	Since it has been assumed that \(\|U_h^0\| \le M_0\), then the local boundedness gives \(\|\mathbf{f}(t_0, U_h^0)\| \le C_b(M_0)\). As such, inequality \eqref{ineq:Wn_pre_stab} yields: 
	\[
	\|W_h^{1}\| \le \|U_h^0\| + \tau\,C_b(M_0) \le M_0 + \tau C_b(M_0). 
	\] 
	The above estimate gives that \(W_h^1\) is bounded and as such, \(\|\mathbf{f}(t_1, W_h^1)\| \le C_b(M_0 + \tau C_b(M_0))\). 
	Thus, we can proceed as follows   
	\begin{align*}
		\|U_h^{1}\| &\le M_0 + \tau\, C_b(M_0) + \frac{\tau}{2}\big(C_b(M_0 + \tau C_b(M_0))+ C_b(M_0)\big), 
		\\ 
		&\le M_0 + \frac{3}{2} \tau C_b(M_0) + \frac{\tau}{2} C_b(M_0 + \tau C_b(M_0)). 
	\end{align*} 
	This establishes the base case. 
	
	The induction hypothesis assumes that \eqref{ineq:Un_stab} holds, i.e. \(\|U_h^n\| \le M_n\), for \(n = 2, 3, \dots, k\). In the induction step which comes next, we seek to estimate \(\|U_h^{k+1}\|\). 
	From \eqref{ineq:Un_pre_stab}, 
	\[
	\|U_h^{k+1}\| \le \|U_h^k\| + \tau\,\|\mathbf{f}(t_k, U^k_h)\| + \frac{\tau}{2}\big(\|\mathbf{f}(t_{k+1},W_h^{k+1})\|+\|\mathbf{f}(t_k, U^k_h)\|\big) 
	\] 
	By the induction hypothesis \(\|U_h^k\| \le M_k\), so the boundedness assumption implies \(\|\mathbf{f}(t_k,U_h^k)\| \le C_b(M_k)\). Then \eqref{ineq:Wn_pre_stab} gives 
	\(\|W_h^{k+1}\|\le M_k + \tau C_b(M_k)\). Consequently, \(\|\mathbf{f}(t_k,W^{k+1}_h)\| \le C_b(M_k + \tau C_b(M_k))\). As such,  
	\[
	\|U_h^{k+1}\| \le M_k + \tau\,C_b(M_k) + \frac{\tau}{2}\big(C_b(M_k + \tau C_b(M_k)) + C_b(M_k)\big) 
	\] 
	Hence, the estimate \eqref{ineq:Un_stab} holds for \(n \in \{1, 2, 3 \dots, N\}\) by induction. 
	Next is to show that the bound is uniform independent of \(n\). 
	
	Define \(\Phi_\tau(r) = 3C_b(r) + C_b (r + \tau C_b(r))\) and observe that: 
	\begin{align*}
		M_n &= M_{n-1} + \frac{\tau}{2}\Phi_\tau(M_{n-1}) = M_{n-2} + \frac{\tau}{2}[\Phi_\tau(M_{n-2}) +\Phi_\tau(M_{n-1})] 
		\\
		&= M_{n-3} + \frac{\tau}{2}[\Phi_\tau(M_{n-3}) + \Phi_\tau(M_{n-2}) +\Phi_\tau(M_{n-1})]. 
	\end{align*}
	Continuing this way yields the following expression 
	\[ 
	M_n = M_{0} + \frac{\tau}{2}\sum_{\ell = 0}^{n-1} \Phi_\tau(M_{\ell}).
	\] 
	The linear growth \(C_b(r) \le c_1 r + c_2\) can be used to obtained 
	\begin{align*}
		\Phi_\tau(r) &\le 3(c_1 r + c_2) + (c_1(r + \tau C_b(r)) + c_2) 
		\\ 
		&\le 3(c_1 r + c_2) + (c_1(r + \tau (c_1 r + c_2)) + c_2) \le c_1 r (4 + c_1 \tau) + (4 + \tau c_1) c_2
	\end{align*} 
	\begin{align*}
		M_n &\le M_0 + \frac{\tau}{2} \sum_{\ell = 0}^{n-1} [(4 + c_1 \tau) c_1 M_\ell + (\tau c_1 + 4) c_2] 
		\\ 
		&\le M_0 + \frac{n \tau}{2}(\tau c_1 + 4) c_2 + \frac{\tau}{2}(4 + c_1 \tau) c_1 \sum_{\ell = 0}^{n-1} M_\ell
	\end{align*} 
	Noting that \(n\tau = t_n \le T\), defining \(C_T = \frac{T}{2}(T c_1 + 4) c_2\) and using the discrete Gr\"onwall inequality \cite[p.~220--221]{Elaydi:2005} yields 
	\[ 
	M_n \le (M_0 + C_T) \exp\left(\frac{n \tau}{2}(4 + c_1 \tau)c_1\right) \le (M_0 + C_T) \exp\left(\frac{T}{2}(4 + c_1 T)c_1\right) =: M_*, 
	\] 
	where \(M_*\) depends on \(c_1, c_2, M_0\) and \(T\) but is independent of \(n, \tau\) and \(N\). Hence the proof.
\end{proof} 

\begin{remark}
	\begin{enumerate} 
		\item 
		The linear-growth assumption in Theorem~\ref{thm:stability_ETD_dimension_split} is not restrictive:
		saturated kinetics $f(u)=\frac{r\,u}{1+\alpha|u|}$ satisfy a global bound $|f(u)|\le r/\alpha$ (so $c_1=0$, $c_2=r/\alpha$).
		Allen--Cahn with $f(u)=u-u^3$, homogeneous Dirichlet boundary data and $|u_0|\le 1$, the maximum principle yields
		$|u(\cdot,t)|\le 1$ for all $t\in[0,T]$, see e.g. \cite{Liao:2020}, hence $|f(u)|\le 2|u|$. So, $(c_1,c_2)=(2,0)$. 
		Thus the hypothesis accommodates several widely-used nonlinearities with explicit constants. 
		
		\item 
		Example~\ref{exp:singular} is used to illustrate that the global linear-growth assumption 
		$C_b(r)\le c_1 r + c_2$ is a sufficient  condition but may not be strictly necessary. 
		In particular, for $f(u)=\tfrac{u}{1-u}$, $0<u<1-\varepsilon$, 
		the constant $c_1$ depends on the upper bound of $u$, 
		so the nonlinearity satisfies only a local linear-growth condition.
	\end{enumerate}
\end{remark}

\subsection{Pad\'e Approximation Error}

The convergence analysis relies on error estimates for the Pad\'e \(P_{0,2}\) approximation of the matrix exponential. 
In the next proposition we present relevant estimates on the matrix-valued Pad'e approximants required for the discretization 
matrices described in Section~\ref{subsec:spatial_Disc}. 

\begin{proposition}\label{prop:Pade-properties}
	Let $P_{0,2}(z)= \frac{2}{2+2z+z^2}$ be the $(0,2)$ Pad\'e approximant of $e^{-z}$,
	and let $A = A_i$, \(i = 1, 2\) be any of the discretization matrices in Section \ref{subsec:spatial_Disc}. 
	There exists a generic constant $C_A>0$
	(depending on the spatial discretization matrices $A_i$, $i=1,2$, but independent of
	$\tau$, $n$, and $N$) such that the following estimates hold.
	Here and below, $C_A$ may denote different constants in different inequalities.

	\begin{enumerate}[label=(\alph*)]
		\item \textbf{Approximation of the matrix exponential.}
		\[
		\|e^{-\tau A} - P_{0,2}(A\tau)\|  \le C_A\tau^3,  \quad \|A^{-2}\,(e^{-\tau A} - P_{0,2}(A\tau))\| \le C_A\tau^3,  
		\]
		\[
		\|I - P_{0,2}(A\tau)\| \le C_A \tau, \quad \|A^{-1}\,(I - P_{0,2}(A\tau))\| \;\le\; C_A\,\tau .
		\]
		\item \textbf{Matrix-Pad\'e bound.}
		\[
		\|P_{0,2}(A\tau)\| \;\le\; 1 + C_A \tau^3.
		\]
		
		\item \textbf{Matrix--exponential product estimate.}
		If $\tau \le 1$, there exists $C_{12} > 0$ such that 
		\[
		\|\,e^{-\tau A_1}e^{-\tau A_2} - P_{0,2}(A_1\tau)P_{0,2}(A_2\tau)\,\| \;\le\; C_{12}\,\tau^3,
		\]
		\[ 
		\|\,P_{0,2}(A_1\tau)P_{0,2}(A_2\tau)\,\| \;\le\; 1 + C_{12}\,\tau^3.
		\] 
		
		\item \textbf{Quadratic remainder.} 
		The combination \(A^{-2}(P_{0,2}(A\tau)-I+\tau A)\) satisfies
		\[
		\|A^{-2}(P_{0,2}(A\tau)-I+\tau A)\| \le C_A\,\tau^2.
		\]
		
	\end{enumerate}
\end{proposition} 

\begin{proof}
	\begin{enumerate}
		\item The discretization matrices \(A_i\) are all symmetric and diagonalizable. As such  we can have \(A_i = Q D Q^\top\), which can be used to obtain 
		\[
		e^{-\tau A} - P_{0,2}(\tau A) = Q \left( e^{-\tau D} - P_{0,2}(\tau D)\right) Q^{\top} = Q \diag\left( \eta(\tau \lambda_\ell)\right) Q^\top, 
		\] 
		where \(\eta(x) =  e^{-x} - P_{0,2}(x)\) and \(\diag(\eta(\tau \lambda_\ell))\) is diagonal matrix with 
		\(\eta(\tau \lambda_1), \dots, \eta(\tau \lambda_m) \) on its diagonal. 
		Using the second Taylor polynomial of $\eta$ at \(x =0\) and the Taylor's theorem, we can have 
		\[ 
		\eta(\tau \lambda_\ell) = -\frac{e^{-\xi_\ell} + P'''_{0,2}(\xi_\ell)}{3!}(\tau \lambda_\ell)^3, \qquad 0 \le \xi_\ell \le \tau \lambda_\ell ,  
		\] 
		where    
		\(P'''_{0,2}(x) = -\frac{48 x (x+1)(x+2)}{(x^2 + 2x + 2)^4}\) is the third derivative of \(P_{0,2}\). 
		By submulitiplicity, 
		\( 
		\|e^{-\tau A} - P_{0,2}(\tau A)\| \le \kappa_\infty (Q) \max_{\ell} \left| \eta(\lambda_\ell \tau) \right| 
		\), where \(\kappa_\infty (Q) = \|Q\| \|Q^\top\| \). 
		Then, 
		\begin{align*}
			\|e^{-\tau A} - P_{0,2}(\tau A)\| &\le \kappa_\infty (Q) \max\limits_{i} \left|\frac{e^{-\xi_i} + P'''_{0,2}(\xi_i)}{3!} \lambda_i^3 \tau^3\right| 
			\\ 
			&\le \frac{1}{6}\kappa_\infty(Q) \sup_{[0, \lambda_\text{max} \tau ]} |e^{-\xi} + P'''_{0,2}(\xi)| \lambda^3_\text{max} \tau^3. 
		\end{align*} 
		The proof is complete by taking \(C_A = \frac{\lambda^3_\text{max}}{6} \kappa_\infty(Q) \sup_{[0, \infty)} |e^{-\xi} + P'''_{0,2}(\xi)| \).  
		
		For the second estimate, we note that 
		
		\noindent \(A^{-2} (e^{-A\tau} - P_{0,2}(A \tau)) = Q \diag(\frac{e^{-\lambda_\ell \tau} - P_{0,2}(\lambda_\ell \tau)}{\lambda_\ell^2}) Q^{\top}\). 
		Then proceeding in the same way as above, yields the desired bound, with
		
		\noindent \(C_A = \frac{\lambda_\text{max}}{6} \kappa_\infty(Q) \sup\limits_{[0, \infty)} |e^{-\xi} + P'''_{0,2}(\xi)| \)
		
		Derivation of the third bound also follows similar arguments as the first, where we replace the definition of $\eta$ by \(\eta(x) = 1 - P_{0,2}(x)\) 
		and take the Taylor polynomial of degree zero. Then the proposed estimate is attained with 
		\(C_A~=~\lambda_\text{max} \kappa_\infty(Q) \sup\limits_{[0, \infty)} |P'_{0,2}(\xi)|\).  
		
		Following similar arguments, the fourth bound can be obtained with \(C_A~=~\kappa_\infty(Q) \sup\limits_{[0, \infty)} |P'_{0,2}(\xi)|\). 
		
		\item  Using triangle inequality together with the first estimate in part (1) and the exponential bound \eqref{ineq:exp_estimate} we have 
		\[
		\| P_{0,2}(\tau A) \| \le \|  e^{-\tau A} \| + \|P_{0,2}(\tau A) - e^{-\tau A}\| \le 1 + C_A \tau^3.
		\] 
		
		\item For the first estimate, we use the exponential estimate \eqref{ineq:exp_estimate}, and the bounds in parts 1 and 2 to obtain
		\begin{align*}
			\|\,e^{-\tau A_1}e^{-\tau A_2} - P_{0,2}(A_1\tau)P_{0,2}(A_2\tau)\| &\le \|e^{-\tau A_1}\| \|e^{-\tau A_2} - P_{0,2}(A_2\tau)\| 
			\\ 
			&+ \|P_{0,2}(A_2\tau)\| \|e^{-\tau A_1} - P_{0,2}(A_1\tau)\| 
			\\ 
			&\le C_{A_2} \tau^3 + (1 + C_{A_2} \tau^3)C_{A_1} \tau^3. 
		\end{align*} 
		With $\tau\le 1$, it holds that \(\tau^6 \le \tau^3\). Then the proof is complete by choosing \(C_{12}~=~C_{A_1}+C_{A_2}+C_{A_1}C_{A_2}\).
		
		For the second estimate, we note that 
		\begin{align*}
			\|P_{0,2}(A_1\tau)P_{0,2}(A_2\tau)\| &\le \|e^{-\tau A_1}e^{-\tau A_2}\| + \|P_{0,2}(A_1\tau)P_{0,2}(A_2\tau) - e^{-\tau A_1}e^{-\tau A_2}\| 
			\\ 
			&\le 1 + C_{12} \tau^3 .
		\end{align*}
		
		\item 
		Using the diagonalization \(A = Q D Q^\top\), it can be shown that 
		\begin{align*}
			A^{-2}\left(P_{0,2}(A\tau) - I + \tau A\right) = Q \diag(\eta(\lambda_\ell \tau)/\lambda^2_\ell) Q^{T}, 
		\end{align*} 
		where $\eta(x) = P_{0,2}(x) - 1 + x$. Using the degree one Taylor polynomial of \(\eta(x)\) and proceeding as in the proof of the first estimate in part (1),  
		we get the desired bound with \(C_A = \kappa_\infty(Q) \max\limits_{[0, \infty)} |P''_{0,2}(\xi)|\).
	\end{enumerate}
\end{proof}

\subsection{Convergence Analysis} 
For the convergence analysis, we assume that the semidiscrete problem 
\eqref{eq:semidiscrete} admits a sufficiently smooth solution 
\(\mathbf{u}_h : [0,T] \to \mathbb{R}^n\).
For convenience, we introduce the notation
\[\Psi(t) := \mathbf{f}(t,\mathbf{u}_h(t)),\] 
from which \(\tilde{\Psi}(t) = e^{-A_2(t_{n+1} - t)} \mathbf{f}(t,\mathbf{u}_h(t))\).
The mapping 
\(\Psi : [0,T] \to \mathbb{R}^n,\; t\mapsto \mathbf{f}(t,\mathbf{u}_h(t))\),
is assumed to be sufficiently smooth on \([0,T]\). And we denote the relevant errors by
\[
\mathcal{E}_h^n := \mathbf{u}_h(t_n) - U_h^n,
\qquad
\widehat{\mathcal{E}}_h^n := \mathbf{u}_h(t_n) - \widehat{U}_h^n.
\]
We begin with a series of lemmas that estimate the errors associated with the
helper schemes~\eqref{eq:ETD_const_fully_discrete} and
\eqref{eq:ETD_const_fully_discrete_Rational_2D}
(as well as their three-dimensional counterpart
\eqref{eq:ETD_const_fully_discrete_Rational_3D}).
These intermediate results are then used to establish the main convergence
theorems corresponding to
\eqref{eq:ETD_linear_fully_discrete},
\eqref{eq:ETD_linear_fully_discrete_Rational_2D},
and
\eqref{eq:ETD_linear_fully_discrete_Rational_3D}.

\begin{lemma}\label{lem:helper_conv} 
	Let $W^n_h$ be the helper scheme \eqref{eq:ETD_const_fully_discrete}. Let $\mathbf{u}_h$ and \(\mathbf{f}\) be the true solution and source term, respectively, in \eqref{eq:semidiscrete}. Suppose $ \Psi  \in \mathcal{C}^1([0,T])$ and \(\mathbf{f}\) satisfies the Lipschitz condition in Assumption \ref{Assump:boundedness_Lipschitz}. Then, for $n = 0, 1, \dots, N-1$ 
	$$
	\|\mathbf{u}_h(t_{n+1}) - W_h^{n+1}\| \leq (1 + L \tau) \|\E_h^{n}\| + \max\limits_{I_{n}} \|\tilde{\Psi}'\| \tau^2
	$$
\end{lemma} 

\begin{proof}
	From the integral forms \eqref{eq:mild_solution_DS} and \eqref{eq:ETD_const_integral}
	$$
	\mathbf{u}_h(t_{n+1}) - W_h^{n+1}  = e^{-A_1 \tau} e^{-A_2 \tau} [\mathbf{u}_h(t_{n}) - U_h^n] + \int_{t_n}^{t_{n+1}} e^{-A_1(t_{n+1} - t)}\left[\tilde{\Psi}(t) - \tilde{\mathbf{f}}(t_n,U^n_h)\right] dt.  
	$$ 
	Using triangle inequality, submultiplicativity, and the exponential bounds in \eqref{ineq:exp_estimate}, 
	$$
	||\mathbf{u}_h(t_{n+1}) - W_h^{n+1}||  \leq ||\E_h^n|| 
	+ \int_{t_n}^{t_{n+1}} \left[\left\|\tilde{\Psi}(t) - \tilde{\Psi}(t_n)\right\| 
	+ \left\|\tilde{\mathbf{f}}(t_n,\mathbf{u}_h(t_n)) - \tilde{\mathbf{f}}(t_n,U^n_h) \right\| \right] dt.  
	$$ 
	For the first term inside the integral, Lagrange error estimate gives  
	$\|\tilde{\Psi}(t) - \tilde{\Psi}(t_n)\| \leq \tau \max\limits_{I_n} \|\tilde{\Psi}'\|$. And for the second term, we apply 
	the Lipschitz continuity and the estimate in \eqref{ineq:exp_estimate}. Then, 
	$$
	\|\mathbf{u}_h(t_{n+1}) - W_h^{n+1}\| \leq \|\mathcal{E}_h^n\| 
	+ \int_{t_n}^{t_{n+1}} \left[\tau \max\limits_{I_n} \|\tilde{\Psi}'\| + L \left\|\mathbf{u}_h(t_n) - U^n_h \right\| \right] dt. 
	$$ 
	Thus, 
	$$
	\|\mathbf{u}_h(t_{n+1}) - W_h^{n+1}\| \leq (1 + L \tau) \|\E_h^n\| +  \max\limits_{I_n} \|\tilde{\Psi}'\| \tau^2, \quad n \in \{0, 1,\dots, N-1\}.  
	$$ 
\end{proof} 

\begin{lemma}\label{lem:helper_conv_Pade}
	Let $\widehat{W}^n_h$ in \eqref{eq:ETD_const_fully_discrete_Rational_2D} be the helper scheme based on the 
	Pad\'e approximant, $P_{0,2}$,   
	and let ${W}^n_h$ be the helper scheme  \eqref{eq:ETD_const_fully_discrete}. 
	Suppose the hypotheses of Theorem \ref{thm:stability_ETD_dimension_split} and Lemma \ref{lem:helper_conv} are satisfied, 
	then, for $n = 0, 1, \dots, N-1$ 
	$$
	||W_h^{n+1} - \widehat{W}_h^{n+1}|| \leq C \tau^3 + (1 + c\tau) ||U^n_h - \widehat{U}^n_h||,
	$$ 
	with constants $C,c>0$ independent of $\tau,n$ (but depending on bounds for $A_1,A_2$, the Lipschitz constant of $f$, and the uniform stability bound).
\end{lemma} 
\begin{proof}
	Recalling the definition of $W_h^n$ in \eqref{eq:ETD_const_fully_discrete_Rational_2D}, 
	we consider $\widehat{W}_h^n$ in a more explicit form by simply replacing the exponentials in \eqref{eq:ETD_const_fully_discrete_Rational_2D} by the Pad\'e without simplification, in other words  
	\begin{equation*}
		\widehat{W}_h^{n+1} = P_{0,2}(A_1 \tau) P_{0,2}(A_2 \tau) \widehat{U}^n_h 
		+  A_1^{-1} \left[I - P_{0,2}(A_1 \tau)\right]P_{0,2}(A_2 \tau) \mathbf{f}(t_n,\widehat{U}_h^n).  
	\end{equation*}  
	Define 
	\begin{equation*}
		\widecheck{W}_h^{n+1} = P_{0,2}(A_1 \tau) P_{0,2}(A_2 \tau) U_h^n 
		+  A_1^{-1} \left[I - P_{0,2}(A_1 \tau)\right]P_{0,2}(A_2 \tau) \mathbf{f}(t_n,U_h^n),   
	\end{equation*} 
	as a comparison function by replacing $\widehat{U}_h^n$ in $\widehat{W}_h^{n+1}$ with $U_h^n$. 
	As such, 
	\begin{equation}\label{eq:err_wn_decompose}
		\left\|W_h^{n+1} -  \widehat{W}_h^{n+1}\right\| \le \left\|W_h^{n+1} -  \widecheck{W}_h^{n+1}\right\| + \left\|\widecheck{W}_h^{n+1} -  \widehat{W}_h^{n+1}\right\|
	\end{equation} 
	Using the triangle inequality, submultiplicity property of the infinity norm as well as the Pad\'e estimates (1) and (3) in Proposition \ref{prop:Pade-properties} the first term on the right hand side can be bounded as follows: 
	\begin{align*}
		\|W_h^{n+1} -  \widecheck{W}_h^{n+1}\| &\le \|e^{-A_1\tau}e^{-A_2\tau} -  P_{0,2}(A_1 \tau) P_{0,2}(A_2 \tau)\| \|U_h^n\| 
		\\ 
		&+\bigg(\|A_1^{-1}(I - P_{0,2}(A_1 \tau)) (e^{-A_2\tau} - P_{0,2}(A_2 \tau))  
		\\ 
		&+  A_1^{-1} (P_{0,2}(A_1 \tau) - e^{-A_1\tau}) e^{-A_2\tau}\|\bigg) \|f(t_n, U_h^n)\| 
		\\ 
		&\le C_{12}\tau^3 \|U_h^n\| + (C_{A_1} \tau \cdot C_{A_2} \tau^3 + \check{C}_{A_1}\tau^3) \|f(t_n, U_h^n)\|.
	\end{align*} 
	Since we assume the hypotheses of Theorem \ref{thm:stability_ETD_dimension_split} are satisfied, then we have \(\|f(t,U_h^n)\| \le c_1 \|U_h^n\| + c_2\) and 
	\(\|U_h^n\| \le M_*\). Taking $\tau \le 1$, we can have 
	\begin{equation}\label{eq:err_wn_term1}
		\left\|W_h^{n+1} -  \widecheck{W}_h^{n+1}\right\| \le C \tau^3, 
	\end{equation} 
	\(C = \max\{C_{12}M_*, C_{A_1} C_{A_2} (c_1 M_* + c_2), \check{C}_{A_1} (c_1 M_* + c_2)\}\). 
	Similarly, for the second term in \eqref{eq:err_wn_decompose}, the estimates (2) and (4) in Proposition \ref{prop:Pade-properties} together with using the triangle inequality, submultiplicity of the infinity norm, and the Lipschitz continuity of $f$, the following inequalities can be attained
	\begin{align}\label{eq:err_wn_term2}
		\|\widecheck{W}_h^{n+1} -  \widehat{W}_h^{n+1}\| &= \|P_{0,2}(A_1 \tau) P_{0,2}(A_2 \tau)\|\|U_h^n - \widehat{U}_h^n\|
		\nonumber 
		\\ 
		&+\bigg(A_1^{-1}\left(I - P_{0,2}(A_1 \tau)\right)P_{0,2}(A_2 \tau)\| \bigg) \|f(t_n, U_h^n) - f(t_n, \widehat{U}_h^n)\| 
		\nonumber
		\\ 
		&\le (1 + C_{12}\tau^3) \|U_h^n - \widehat{U}_h^n\| + L C_{A_1} \tau (1 + C_{A_2}\tau^3) \|U_h^n - \widehat{U}_h^n\| 
		\nonumber
		\\  
		\nonumber
		&\leq  (1 + C_{12}\tau^3 + L C_{A_1}\tau + L C_{A_1} C_{A_2}\tau^4) \|U^n_h - \widehat{U}^n_h\| 
		\\ 
		&\leq (1 + c\tau) \|U^n_h - \widehat{U}^n_h \|,  
	\end{align} 
	\(c = \max(C_{12}, L C_{A_2}, L C_{A_1} C_{A_2})\), where we have assumed \(\tau \le 1\) to have \(\tau^4 \le \tau^3 \le \tau\).
\end{proof}

The next result gives the error estimate for the dimension split ETD2RK-DS scheme \eqref{eq:ETD_linear_fully_discrete}. 
\begin{theorem}\label{thm:convergence ETD dimension split}
	Let $U^n_h$ be the ETD solution defined by the dimension split scheme \eqref{eq:ETD_linear_fully_discrete}. 
	Let $\mathbf{u}_h$ be the solution of \eqref{eq:semidiscrete}. Suppose $\Psi \in \mathcal{C}^2([0,T])$ and that \(\mathbf{f}\) satisfies 
	the Lipschitz condition in Assumption \ref{Assump:boundedness_Lipschitz}. Then, for $n = 0, 1, \dots, N$ 
	$$
	\|\mathbf{u}_h(t_{n}) - U_h^{n}\| \leq t_{n} e^{2 c t_{n}} \max\limits_{[0,T]} \left(L\|\tilde{\Psi}'(t)\| +  \frac{1}{8}\|\tilde{\Psi}^{(2)}(t)\|\right)\tau^2, 
	$$ 
	\(c = \max(2L, L^2)\).
\end{theorem} 
\begin{proof}
	As a comparison function, define 
	$$
	\check{\mathbf{p}}(t) = \tilde{\mathbf{f}}(t_n, \mathbf{u}_h(t_n)) 
	+ \frac{t - t_n}{\tau} \left[\mathbf{f}(t_{n+1}, \mathbf{u}_h(t_{n+1})) - \tilde{\mathbf{f}}(t_{n}, \mathbf{u}_h(t_{n})) \right]. 
	$$ 
	Then, from the integral forms \eqref{eq:mild_solution_DS} and \eqref{eq:ETD_linear_integral}, we can have 
	\begin{align*}
		\E_h^{n+1} &= e^{-A_1 \tau} e^{-A_2 \tau} \E_h^{n}
		\\
		& + \int_{t_n}^{t_{n+1}} e^{-A_1(t_{n+1} - t)}\left[\tilde{\Psi}(t) - \check{\mathbf{p}}(t) + \check{\mathbf{p}}(t) - \mathbf{p}(t)\right] dt
	\end{align*} 
	from which using submultiplicity, triangle inequality, and the exponential estimate \eqref{ineq:exp_estimate}: 
	\begin{align}\label{eq:err_integral_form_ETD2Rk-DS}
		\left\|\E_h^{n+1}\right\| \le  \|\E_h^{n}\|
		+ \int_{t_n}^{t_{n+1}}  \left[\left\|\tilde{\Psi}(t) - \check{\mathbf{p}}(t)\right\| + \|\check{\mathbf{p}}(t) - \mathbf{p}(t)\|\right] dt.
	\end{align}  
	Since $\Psi \in \mathcal{C}^2([0,T])$, then the Lagrange error estimate gives 
	
	\noindent $\left\|\tilde{\Psi}(t) - \check{\mathbf{p}}(t)\right\|~\le~\frac{1}{8} \max\limits_{I_n} \|\tilde{\Psi}^{(2)}(t)\| \tau^2$. 
	
	And for the second term in the integral, we use the estimate in Lemma \ref{lem:helper_conv} and the Lipschitz continuity of \(f\) to reach 
	\begin{align*}
		\|\check{\mathbf{p}}(t) - \mathbf{p}(t)\| &\le L\|\E^n_h\| + L\|\mathbf{u}_h(t_{n+1}) - W_h^{n+1}\| 
		\leq L(2 + L \tau)\|\E^n_h\| + L \max_{I_n}\|\tilde{\Psi}'\|\tau^2 
		\\ 
		&\leq c(1 + \tau) \|\E^n_h\| + L \max_{I_n}\|\tilde{\Psi}'\| \tau^2, \quad c = \max(2L, L^2), 
	\end{align*} 
	Putting the Lagrange error estimate and the above bound in into \eqref{eq:err_integral_form_ETD2Rk-DS} gives 
	\begin{align*}
		\left\|\E_h^{n+1}\right\| &\leq \|\E_h^{n}\|
		+ c\tau(1+\tau)\left\|\E_h^{n}\right\| + \max\limits_{I_n} (L \|\tilde{\Psi}'(t)\| + \frac{1}{8}\|\tilde{\Psi}^{(2)}(t)\|) \tau^3
		\\
		&\le (1 + c \tau + c \tau^2)\left\|\E_h^{n}\right\| + M \tau^3, 
		\qquad M = \max\limits_{[0,T]} (\|L \tilde{\Psi}'(t)\| + \frac{1}{8}\|\tilde{\Psi}^{(2)}(t)\|).
	\end{align*} 
	Enumeration the above recursive relation shows that
	\begin{align*}
		\left\|\E_h^{1}\right\| &\leq (1 + c\tau + c\tau^2)\|\E_h^{0}\| + M \tau^3
		\\ 
		\left\|\E_h^{2}\right\| &\leq (1 + c\tau + c\tau^2)\|\E_h^{1}\| + M\tau^3
		\leq (1 + c\tau + c\tau^2)^2 \|\E_h^0\| +  (1 + c \tau + c \tau^2) M\tau^3 + M\tau^3 .
	\end{align*} 
	Continuing this way  and noting \(\E^0_h = \mathbf{0}\) 
	\begin{align*}
		\left\|\E_h^{n}\right\| \le M \tau^3\sum_{\ell = 0}^{n-1} (1 + c \tau + c \tau^2)^{\ell}  
	\end{align*} 
	Using the inequalities \(\sum\limits_{\ell = 0}^{n-1} (1 + \alpha)^\ell \le n(1 + \alpha)^{n-1}, \,\,\, \alpha \ge 0\), and $(1 + x)^n \le e^{nx}, \,\, x \ge 0$, the following estimates can be obtained 
	\begin{align*}
		\left\|\E_h^{n}\right\| \le M\tau^3 n (1 + c\tau + c\tau^2)^{n-1} \leq M t_n e^{cn(\tau + \tau^2)} \tau^2. 
	\end{align*}
	The proof is complete by noting that \(e^{c n(\tau + \tau^2)} \le e^{2 c n\tau} = e^{2 c t_n} \). The last inequality assumes \(\tau\) small (suffices that \(\tau \le 1\)).
\end{proof}
\noindent This establishes the second-order temporal convergence of the underlying ETD2RK--DS scheme. 

\noindent Next, result gives an error estimate for the Pad\'e based implementation of the dimension split ETD scheme. 
\begin{theorem}\label{thm:convergence ETD dimension split rational approx}
	Let $\widehat{U}^n_h$ be the dimension-split Pad\'e-based ETD2RK-DS scheme defined by \eqref{eq:ETD_linear_fully_discrete_Rational_2D} (or  \eqref{eq:ETD_linear_fully_discrete_Rational_3D}), and let $\mathbf{u}_h$ be the solution of \eqref{eq:semidiscrete}. 
	Suppose the hypotheses of Theorems \ref{thm:stability_ETD_dimension_split} and \ref{thm:convergence ETD dimension split} are satisfied.
	Then, for $n = 0, 1, \dots, N$ 
	$$
	||\mathbf{u}_h(t_{n}) - \widehat{U}_h^{n}|| \leq  t_n e^{c_1 t_n} \max_{[0, T]} (\|\tilde{\Psi}'(t)\| + \frac{1}{8}\|\tilde{\Psi}^{(2)}(t)\|)\tau^2 + \tilde{c} e^{\check{c}t_n} \tau .
	$$
\end{theorem} 

\begin{proof}
	By triangle inequality, 
	\begin{equation*}
		\|\widehat{\E}^{n+1}_h\| = \|\mathbf{u}(t_{n+1}) - \widehat{U}_h^{n+1}\| \leq \|\mathbf{u}(t_{n+1}) - U_h^{n+1}\| + \|U_h^{n+1} - \widehat{U}_h^{n+1}\|. 
	\end{equation*} 
	From the estimate in Theorem \ref{thm:convergence ETD dimension split} 
	
	\noindent \(\|\mathbf{u}(t_{n+1}) - U_h^{n+1}\|~\le~t_{n+1} e^{c_1 t_{n+1}}~\max\limits_{[0, T]} (\|\tilde{\Psi}'(t)\|~+~\frac{1}{8}\|\tilde{\Psi}^{(2)}(t)\|)\tau^2\), 
	\begin{equation}\label{eq:err_decompose}
		\|\widehat{\E}^{n+1}_h\| \leq  t_{n+1} e^{c_1 t_{n+1}} \max_{[0, T]} \left(\|\tilde{\Psi}'(t)\| + \frac{1}{8}\|\tilde{\Psi}^{(2)}(t)\|\right)\tau^2 + \|U_h^{n+1} - \widehat{U}_h^{n+1}\|. 
	\end{equation} 
	For the second term in the right hand side above, consider $\widehat{U}_h^n$ in the form 
	\begin{align*}
		\widehat{U}_h^{n+1} &= \widehat{W}_h^{n+1} 
		\nonumber \\ 
		&+ \left( A_1^{-1}\tau -  A_1^{-2} + A_1^{-2} P_{0,2}(A_1\tau)\right) 
		\frac{\mathbf{f}(t_{n+1},\widehat{W}_h^{n+1}) - P_{0,2}(A_2\tau) \mathbf{f}(t_n,\widehat{U}_h^n)}{\tau},  
	\end{align*} 
	Then, the difference \(U^{n+1}_h - \widehat{U}_h^{n+1}\) can be written as 
	\begin{align}\label{eq:PadeTerm_Error}
		U^{n+1}_h - \widehat{U}_h^{n+1} &= W_h^{n+1} - \widehat{W}_h^{n+1} 
		\nonumber \\ 
		&+ \left( A_1^{-1}\tau -  A_1^{-2} + A_1^{-2} e^{-A_1\tau}\right) 
		\frac{\mathbf{f}(t_{n+1}, W_h^{n+1}) - e^{-A_2\tau} \mathbf{f}(t_n,U_h^n)}{\tau}
		\nonumber \\ 
		&- \left( A_1^{-1}\tau -  A_1^{-2} + A_1^{-2} P_{0,2}(A_1\tau)\right) 
		\frac{\mathbf{f}(t_{n+1},W_h^{n+1}) - P_{0,2}(A_2\tau) \mathbf{f}(t_n,U_h^n)}{\tau}
		\nonumber \\ 
		&+ \left( A_1^{-1}\tau -  A_1^{-2} + A_1^{-2} P_{0,2}(A_1\tau)\right) 
		\frac{\mathbf{f}(t_{n+1},W_h^{n+1}) - P_{0,2}(A_2\tau) \mathbf{f}(t_n,U_h^n)}{\tau}
		\nonumber \\ 
		&- \left( A_1^{-1}\tau -  A_1^{-2} + A_1^{-2} P_{0,2}(A_1\tau)\right) 
		\frac{\mathbf{f}(t_{n+1},\widehat{W}_h^{n+1}) - P_{0,2}(A_2\tau) \mathbf{f}(t_n,\widehat{U}_h^n)}{\tau}  
		\nonumber \\ 
		&= W_h^{n+1} - \widehat{W}_h^{n+1} + \mathbb{T}_1 + \mathbb{T}_2 + \mathbb{T}_3,
	\end{align}  
	where 
	\[
	\mathbb{T}_1 = \frac{1}{\tau}A_1^{-2}\left(e^{-A_1 \tau} - P_{0,2}(A_1 \tau)\right) f(t_{n+1}, W_h^{n+1}), 
	\]
	\begin{align*}
		\mathbb{T}_2 &=\frac{1}{\tau} \bigg(\left(A_1^{-1} \tau - A_1^{-2} + A_1^{-2} P_{0,2}(A_1 \tau) \right)P_{0,2}(A_2 \tau) - \left(A_1^{-1} \tau - A_1^{-2} + A_1^{-2} e^{-A_1 \tau} \right)e^{-A_2 \tau}\bigg) \mathbf{f}(t_n, U_h^n)   
		\\ 
		&= \frac{1}{\tau} \bigg(A_1^{-2}\left(A_1 \tau - I + P_{0,2}(A_1 \tau) \right)(P_{0,2}(A_2 \tau) - e^{-A_2 \tau}) + A_1^{-2}(P_{0,2}(A_1 \tau) - e^{-A_1 \tau})e^{-A_2 \tau}\bigg) \mathbf{f}(t_n, U_h^n)
	\end{align*} 
	and  
	\begin{align*}
		\mathbb{T}_3 &= \frac{1}{\tau}\left(A_1^{-1}\tau - A_1^{-2}(I - P_{0,2}(A_1 \tau))\right) \bigg[\left(\mathbf{f}(t_{n+1}, W_h^{n+1}) - \mathbf{f}(t_{n+1}, \widehat{W}_h^{n+1})\right) 
		\\
		&+ P_{0,2}(A_2 \tau)\left(\mathbf{f}(t_n, \widehat{U}_h^n) - \mathbf{f}(t_n, U_h^n)\right)\bigg]. 
	\end{align*} 
	Since the hypothesis of this theorem assumes the stability result holds, then there exist \(M_1, M_2 > 0\) such that \(\|W^{n}\| \le M_1\) and \(\|U^{n}\| \le M_2\) for each \(n\). Consequently, \(\|f(t_n, W^n_h)\| \le C_b(M_1)\) and \(\|f(t_n, U^n_h)\|\le C_b(M_2)\). 
	
	Without loosing generality, assume $\tau \le 1$. Now, applying the second estimate in Proposition \ref{prop:Pade-properties}(1):
	\[
	\|\mathbb{T}_1 \| \le \frac{1}{\tau} \widetilde{C}_{A_1} \tau^3 C_b(M_1) = \widetilde{C}_{A_1} C_b(M_1) \tau^2  
	\] 
	Similarly, by the stability assumption together with parts 1 and 4 of Proposition \ref{prop:Pade-properties} 
	\[
	\|\mathbb{T}_2 \| \le \frac{1}{\tau}\left( \widecheck{C}_{A_1} \tau^2 C_{A_2} \tau^3 + \widetilde{C}_{A_1}\tau^3 \right) C_b(M_2) \le (\widecheck{C}_{A_1} C_{A_2} + \widetilde{C}_{A_1}) C_b(M_2) \tau^2. 
	\] 
	Finally, for \(\mathbb{T}_3\), we apply the Lipschitz continuity in Assumption \ref{Assump:boundedness_Lipschitz} together with parts 2 and 4 of Proposition 
	\ref{prop:Pade-properties}, and the estimate in Lemma \ref{lem:helper_conv_Pade} to obtain 
	\begin{align*}
		\|\mathbb{T}_3\| &\le \widecheck{C}_{A_1} \tau [\|W_h^{n+1} - \widehat{W}_h^{n+1}\| + (1 + C_{A_2}\tau^3) \|U_h^n - \widehat{U}_h^n\|] 
		\\ 
		& \le \widecheck{C}_{A_1} \tau [L(C_w\tau^3 + (1+c_w \tau)\|U_h^{n} - \widehat{U}_h^{n}\|) + (1 + C_{A_2}\tau^3) L \|U_h^n - \widehat{U}_h^n\|]
		\\ 
		&\le  \widecheck{C}_{A_1} L C_w\tau^4 + \widecheck{C}_{A_1} L(2 + c_w  + C_{A_2})\tau\|U_h^{n} - \widehat{U}_h^{n}\| . 
	\end{align*}
	Now, using the above estimates on \(\|\mathbb{T}_1 \|, \|\mathbb{T}_2 \|\), and \(\|\mathbb{T}_3\|\) together with Lemma \ref{lem:helper_conv_Pade} in \eqref{eq:PadeTerm_Error}: 
	\begin{align*}
		\|U^{n+1}_h - \widehat{U}^{n+1}_h\| &\le \tilde{c} \tau^2 + \check{c}\tau\|U^{n}_h - \widehat{U}^{n}_h\|, 
	\end{align*} 
	where 
	\begin{align*}
		\tilde{c} &= \widetilde{C}_{A_1} C_b(M_1) +  (\widecheck{C}_{A_1} C_{A_2} + \widetilde{C}_{A_1}) C_b(M_2) + \widecheck{C}_{A_1} L C_w, 
		\quad \check{c} = \widecheck{C}_{A_1} L(2 + c_w  + C_{A_2}). 
	\end{align*} 
	Proceeding recursively, we have  
	\[ 
	\|U^{n}_h - \widehat{U}^{n}_h\| \le (\check{c}\tau)^n \|U^{0}_h - \widehat{U}^{0}_h\| + \tilde{c}\tau^2 \sum_{\ell = 0}^{n-1} (\check{c}\tau)^{\ell} .
	\] 
	Noting \(U_h^0 = \widehat{U}^0_h = \mathbf{u}_h(0)\), the following estimate can be obtained: 
	\begin{equation}\label{eq:Pade_Err_Contribution}
		\|U^{n}_h - \widehat{U}^{n}_h\| \le \tilde{c}\tau^2 \sum_{\ell = 0}^{n-1} (\check{c}\tau)^{\ell} \le \tilde{c}\tau^2 \sum_{\ell = 0}^{n-1} (1 + \check{c}\tau)^{\ell} 
		\le  \tilde{c}\tau^2 n (1 + \check{c}\tau)^{n-1} \le  \tilde{c} t_n e^{\check{c}t_n} \tau.
	\end{equation} 
	We have used the inequalities \(\sum\limits_{\ell = 0}^{n-1} (1 + \alpha)^\ell \le n(1 + \alpha)^{n-1}, \,\,\, \alpha \ge 0\), and $(1 + x)^n \le e^{nx}, \,\, x \ge 0$, to achieve the above estimate. Putting \eqref{eq:Pade_Err_Contribution} in \eqref{eq:err_decompose} completes the proof. 
\end{proof} 

\begin{remark}
	The error bound in Theorem~\ref{thm:convergence ETD dimension split rational approx}
	contains an $\mathcal{O}(\tau)$ contribution arising from the Pad\'e approximation of
	the exponential. This term is based on a uniform worst--case estimate of the Pad\'e defect
	and may not be sharp. In particular, the experiments in Section~\ref{Sec:Numerics} 
	indicate that an overall temporal accuracy of order $\mathcal{O}(\tau^{2})$ in the test cases. 
\end{remark}

\section{Numerical Experiment}\label{Sec:Numerics}
In this section, three numerical examples are presented to verify the theoretical
results and demonstrate the practical performance of the proposed methods.
In particular, we examine both the convergence and stability properties of
the dimension-splitting ETD2RK-DS scheme, and we compare the efficiency of
its Sylvester–equation implementation with a conventional LU–factorization
approach, especially for large-scale problems.

The first example is a nonautonomous Allen–Cahn type equation with a
manufactured exact solution, used to quantify accuracy and verify the
second-order convergence rate.
The second example involves a locally bounded source that violates the linear
growth assumption in
Theorem~\ref{thm:stability_ETD_dimension_split}, thereby showing that this
condition is sufficient but not strictly necessary for stability.
Finally, we consider the FitzHugh–Nagumo model to illustrate the performance
of the proposed scheme and its implementation strategy for systems of
reaction–diffusion equations.

To measure the temporal accuracy, we evaluate the numerical error on a fixed
coarse time grid.  Let
\(
\mathcal{S}_N = \{t_n = nT/N : n = 0,1,\dots,N\},
\)
and fix a coarse set $\mathcal{S}_{16}$ in two dimensions and $\mathcal{S}_{10}$
in three dimensions.
Since $\mathcal{S}_{16} \subseteq \mathcal{S}_N$ whenever $N$ is a multiple of $16$
(and analogously $\mathcal{S}_{10} \subseteq \mathcal{S}_N$ whenever $N$ is a
multiple of $10$), we compute the $\ell^\infty$ error at the common time points.
Specifically,
\[ 
E(N) = \max\limits_{t_n \in S_{N} \cap S_{16}} \|u(x_i, y_j, t_n) - U^n_{h,ij}\|, \text{ and } 
E(N) = \max\limits_{t_n \in S_{N} \cap S_{10}} \|u(x_i, y_j, z_k, t_n) - U^n_{h,ijk}\|
\] for the two and three dimensional cases, respectively. This ensures that all $E(N)$ values are computed at identical physical times.
The experimental order of convergence (EOC) is then defined as  \(\log_2\left(\frac{E(N)}{E(2N)}\right)\).  

Further, we assess the performance of our ETD2RK-DS scheme as well as the implementation algorithm developed here relative to the Pade (\(P_{0,2}\)) based ETD2RK scheme without dimension splitting, given by  
\begin{equation*}
	\widehat{W}_h^{n+1} 
	= Q_1(\tau A_h) \widehat{U}_h^n + \tau Q_2(\tau A_h) \mathbf{f}(t_n, \widehat{U}_h^n),
\end{equation*}
and
\begin{equation*}
	\widehat{U}_h^{n+1}  = \widehat{W}_h^{n+1}  + \tau Q_3(\tau A_h) \left[ \mathbf{f}(t_{n+1}, \widehat{W}_h^{n+1}) - \mathbf{f}(t_n, \widehat{U}_h^n) \right],
\end{equation*} 
This nonsplit scheme provides a natural baseline for quantifying both
the computational efficiency and the accuracy of the proposed
dimension-splitting formulation. 

We also examine the sensitivity of the linear solvers (LU and the Sylvester reformulation) to the degree of the 
Pad\'e approximant used in the ETD2RK-DS scheme. In addition to the $(0,2)$ Pad\'e formula, we consider the 
higher-order approximant
\[
P_{0,4}(x) = 24\,(x^{4} + 4x^{3} + 12x^{2} + 24x + 24I)^{-1},
\]
which leads to the Pad\'e-based ETD2RK-DS method
\begin{equation}
	\widehat{W}_h^{\,n+1}
	= R_1(\tau A_1)R_1(\tau A_2)\widehat{U}_h^n
	+ \tau\,R_2(\tau A_1)R_1(\tau A_2)\mathbf{f}(t_n,\widehat{U}_h^n),
\end{equation}
\begin{equation}
	\widehat{U}_h^{\,n+1}
	= \widehat{W}_h^{\,n+1}
	+ \tau\,R_3(\tau A_1)\!\left[
	\mathbf{f}(t_{n+1},\widehat{W}_h^{\,n+1})
	- R_1(\tau A_2)\mathbf{f}(t_n,\widehat{U}_h^n)
	\right],
\end{equation}
where
\[
R_1(x)=P_{0,4}(x), \qquad
R_2(x)=(x^3 + 4x^2 + 12x + I)(x^{4} + 4x^{3} + 12x^{2} + 24x + 24I)^{-1},
\]
\[
R_3(x)=(x^3 + 3x^2 + 8x + 12I)(x^{4} + 4x^{3} + 12x^{2} + 24x + 24I)^{-1}.
\]
These rational functions admit the partial-fraction representation
\[
R_\ell(x)
= 2\,\mathrm{Re}\!\left(
\sum_{l=1}^{2} r^{(\ell)}_l (x - s_l I)^{-1}
\right),
\qquad \ell=1,2,3.
\]
This comparison allows us to assess how the computational efficiency of the LU and Sylvester implementations depends on 
the Pad\'e degree. Higher-order ETD schemes such as ETD3RK and ETD4RK typically require Pad\'e approximants of higher 
degree to maintain their design order, and the results presented here illustrate which linear-solver strategy scales more 
favorably as the Pad\'e order increases.

\text{Computational Environment:} All numerical experiments were implemented in Python on a workstation equipped 
with a 4th‑Generation Intel\textregistered\ Xeon\textregistered\ w5-3425 processor (12 cores, 3.20~GHz) and 32~GB of 
RAM. As noted in Remark~\ref{rem:algorithms}, nested loops were avoided by implementing the spectral solver (as well 
as the LU-based reference implementation) in matrix--matrix form, enabling efficient use of optimized BLAS routines.

\begin{example}[Multidimensional nonautonomous Allen-Cahn type equations] \label{exp:Allen-Cahn}  
	
	For testing purposes, consider the model problem on the unit domains $\Omega=(0,1)^d$, $d=2,3$, with final time \(T = 1\) and 
	initial data 
	\[
	u_0(x,y)=\sin(\pi x)\sin(\pi y) \qquad u_0(x,y,z)=\sin(\pi x)\sin(\pi y)\sin(\pi z). 
	\]
	In addition, we take \(q = 0\) and choose the nonlinear source term  
	as 
	\[ 
	f(\mathbf{x},t,u) = u(1 - u^2) + \psi(\mathbf{x}, t),  
	\] 
	where the forcing term $\psi(\mathbf{x}, t)$ is chosen so that the true solution
	of the problem is $u(\mathbf{x}, t) = e^{-\lambda t}u_0(\mathbf{x})$. 
	As such, in two dimension we can obtain 
	$$
	\psi(x,y,t) = (-\lambda  u_0(x,y) + 2\pi^2 u_0(x,y)) e^{-\lambda t} 
	-  e^{-\lambda t} u_0(x,y)(1 - (e^{-\lambda t} u_0(x,y))^2). 
	$$ 
	An analogous expression is obtained in the three dimensions.
\end{example} 

Taking $\lambda = 1$, numerical results obtained by solving this problem by the Pad\'e based dimension split schemes with 
$m_x = m_y = 512$ and $m_x = m_y = m_z = 80$ are presented in Tables \ref{tab:Allen-Cahn2D} and 
\ref{tab:Allen-Cahn3D}, respectively. It is observed that an optimal $\mathcal{O}(\tau^2)$ temporal convergence 
rate is attained in all cases. This suggests that the pessimistic $\mathcal{O}(\tau)$ term in 
the error estimate in Theorem \ref{thm:convergence ETD dimension split rational approx} does not dominantly influence the 
observed numerical results. These observations suggest that, under sufficient regularity of the initial
data, the solution, and the source term, the Pad\'e-based dimension-splitting
scheme can be expected to behave as a second-order method in practice, despite the presence of
the $\mathcal{O}(\tau)$ defect term in the theoretical error estimate. 

Plots of the time evolution of the pointwise relative error \(\frac{\|u(t_n) - U^n_h\|}{\|u(t_n)\|}\) are 
given in Figure \ref{fig:Allen-Cahn_Relative_Error} for different number of time mesh elements \(N\). 
It can be observed that the error levels are reduced by approximately a factor of four when $N$ is doubled, 
which also reflects a numerical second-order temporal accuracy, observed for the absolute error.

In addition, the numerical outputs for the LU and Sylvester-based implementation approaches are presented for 
verification purposes. The identical values observed from the two implementation techniques demonstrate that the spectral decomposition 
approach does not introduce numerical instabilities to the scheme. Also, the CPU timings show that the Sylvester-based 
implementation becomes notably faster than the LU approach as the number of mesh points (\(N\)) increases.

Further, when the Pad\'e \(P_{0,4}\) is employed the number of linear systems to be solved formally doubles. 
However,  Table~\ref{tab:Allen-Cahn2D} shows that while the runtime for LU  approximately doubles, the 
Sylvester-based runtime increases only mildly. 
This is because, in the Sylvester-based approach, the number of dominant
computational tasks involving matrix--vector multiplications remains unchanged,
with only inexpensive elementwise divisions added. In contrast, the LU-based
implementation requires solving twice as many linear systems. 
This highlights that the Sylvester-based implementation approach developed 
here is computationally robust to the type Pad\'e approximant used. 
These observations indicate that the Sylvester-based 
implementation may be particularly attractive for higher-order ETD schemes, such as ETD3RK and ETD4RK, which often require higher-degree Pad\'e 
approximants to attain their design accuracy.  

Finally, Table \ref{tab:AC_ETD2RK_VS_ETD2RK-DS_CPU} compares the dimension-split ETD2RK-DS scheme with the non-split ETD2RK method. 
It is seen that the error values \(E(N)\) are quite different. This can be expected as the non-split scheme is constructed based on polynomial 
approximations of the source function \(\mathbf{f}\) only, whereas the splitting scheme involves the approximation of the product  
(\(e^{-A_1(t_{n+1} - t)} \mathbf{f}\)). In terms of efficiency, however, the advantage is decisive: the dimension-split scheme achieves 
a $10^{-4}$ accuracy within $0.22$~s, whereas the nonsplit solver requires over $27$~s to reach a similar error.
This demonstrates the substantial computational gain afforded by the dimension-splitting. 

\begin{table}[]
	\begin{tabular}{clllllll}
		\multicolumn{1}{l}{}         &     & \multicolumn{3}{l}{Sylvester} & \multicolumn{3}{l}{LU}  \\ 
		\cmidrule(lr){3-5} \cmidrule(lr){6-8}
		\multicolumn{1}{l}{}         & $N$ & $E(N)$     & EOC    & CPU    & $E(N)$   & EOC  & CPU   \\ 
		\thickhline
		\multirow{5}{*}{\(P_{0,2}\)} & 16  & 5.05E-02   & 1.72   & 0.81   & 5.05E-02 & 1.72 & 1.14  \\
		& 32  & 1.53E-02   & 1.88   & 1.66   & 1.53E-02 & 1.88 & 2.27  \\
		& 64  & 4.15E-03   & 1.94   & 3.43   & 4.15E-03 & 1.94 & 4.53  \\
		& 128 & 1.08E-03   & 1.97   & 6.85   & 1.08E-03 & 1.97 & 9.00  \\
		& 256 & 2.76E-04   &        & 13.82  & 2.76E-04 &      & 18.02 \\
		\thickhline
		\multirow{5}{*}{\(P_{0,4}\)} & 16  & 2.30E-02   & 1.93   & 1.03   & 2.30E-02 & 1.93 & 2.04  \\
		& 32  & 6.06E-03   & 1.99   & 2.23   & 6.06E-03 & 1.99 & 4.03  \\
		& 64  & 1.53E-03   & 2.00   & 4.44   & 1.53E-03 & 2.00 & 8.05  \\
		& 128 & 3.83E-04   & 1.98   & 8.92   & 3.83E-04 & 1.98 & 17.33 \\
		& 256 & 9.73E-05   &        & 19.52  & 9.73E-05 &      & 36.49
	\end{tabular} 
	\caption{Example \ref{exp:Allen-Cahn}: temporal convergence and CPU-time comparison for 
		the LU-factorization and Sylvester-equation implementations of the
		Pad\'e--ETD2RK-DS scheme for the two-dimensional Allen--Cahn equation.}\label{tab:Allen-Cahn2D}
\end{table} 
\begin{table}[] 
	\centering
	\begin{tabular}{l|lll|lll}
		& \multicolumn{3}{c}{Sylvester} & \multicolumn{3}{c}{LU}  \\  
		\cmidrule(lr){2-4} \cmidrule(lr){5-7} 
		$N$ & $E(N)$     & EOC    & CPU    & $E(N)$   & EOC  & CPU   \\
		\thickhline
		10  & 2.96E-01   & 1.51   & 1.72   & 2.96E-01 & 1.51 & 1.84  \\
		20  & 1.04E-01   & 1.78   & 3.45   & 1.04E-01 & 1.78 & 3.70  \\
		40  & 3.01E-02   & 1.90   & 6.70   & 3.01E-02 & 1.90 & 7.36  \\
		80  & 8.05E-03   & 1.91   & 13.16  & 8.05E-03 & 1.91 & 15.24 \\
		160 & 2.14E-03   &        & 26.36  & 2.14E-03 &      & 32.49
	\end{tabular} 
	\caption{Example \ref{exp:Allen-Cahn} temporal convergence and CPU-time comparison for the
		Sylvester and LU-factorization implementations of the
		ETD2RK-DS scheme for the three-dimensional Allen--Cahn equation.}\label{tab:Allen-Cahn3D}
\end{table} 
\begin{table}[]
	\begin{tabular}{lllll}
		& \multicolumn{2}{c}{ETD2RK-DS (LU)} & \multicolumn{2}{c}{ETD2RK (LU)} \\ 
		\cmidrule(lr){2-3} \cmidrule(lr){4-5} 
		$N$ & $E(N)$             & CPU           & $E(N)$           & CPU          \\
		16  & 2.05E-02           & 0.04          & 5.82E-04         & 27.52        \\
		32  & 5.75E-03           & 0.07          & 1.90E-04         & 29.11        \\
		64  & 1.52E-03           & 0.11          & 7.42E-05         & 36.23        \\
		128 & 4.16E-04           & 0.22          & 4.76E-05         & 53.61       
	\end{tabular} 
	\caption{Example \ref{exp:Allen-Cahn}: CPU-time comparison between the dimension-splitting
		ETD2RK-DS scheme and the nonsplit Pad\'e--ETD2RK scheme in solving the two-dimension problem. 
	}\label{tab:AC_ETD2RK_VS_ETD2RK-DS_CPU}
\end{table} 

\begin{figure}
	\centering
	\includegraphics[width=0.49\textwidth]{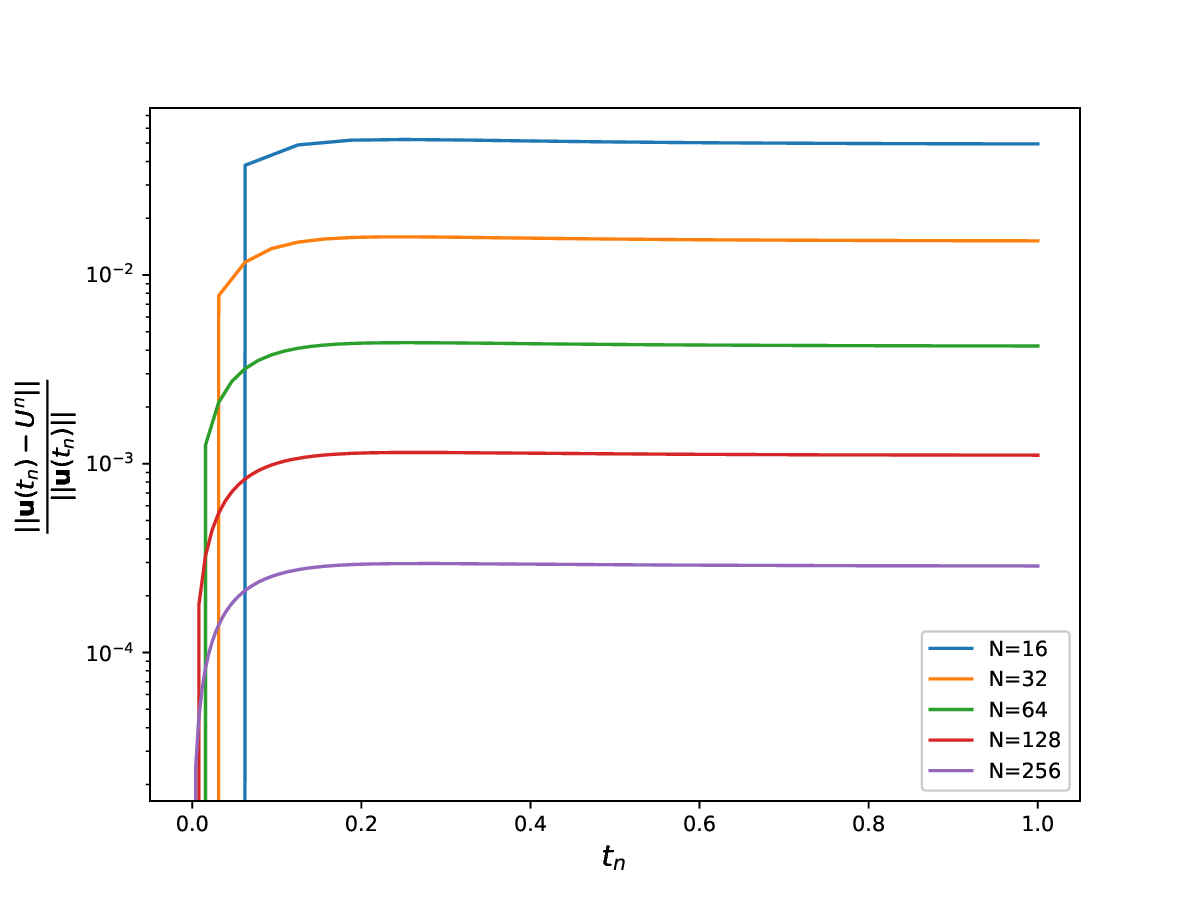} 
	\includegraphics[width=0.49\textwidth]{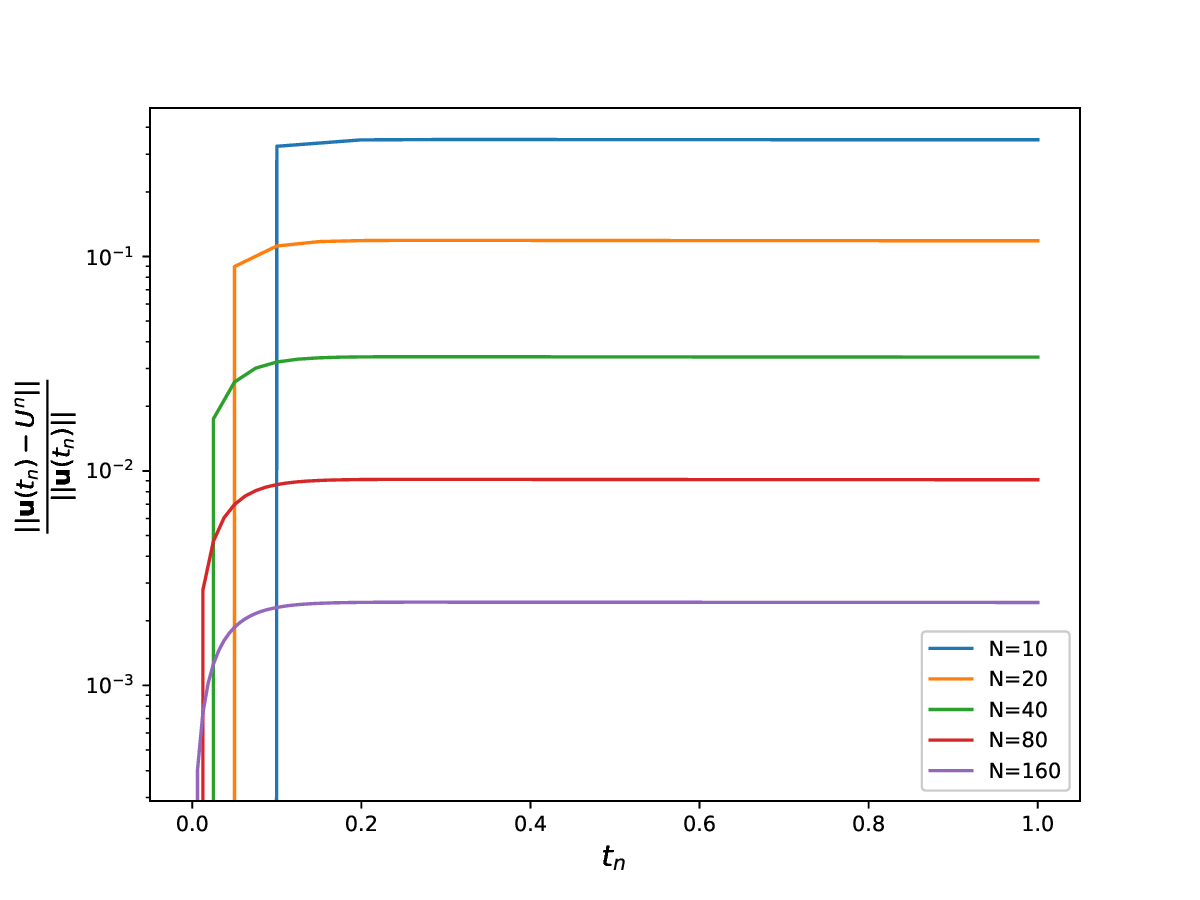}
	\caption{Example \ref{exp:Allen-Cahn}, comparison of the pointwise relative approximation error, 
		for different values of $N$, 2D (left), 3D (right)}
	\label{fig:Allen-Cahn_Relative_Error}
\end{figure}

\begin{example}\label{exp:singular}(Non-globally bounded source) 
	Consider the model problem \eqref{eq:model} with as source term 
	\[
	f(x,y,t,u)=f(u)=\rho\,\frac{u}{1-u},\qquad \rho>0,
	\]   
	and set $q=1$. In addition, we take the initial data $u_0(x,y) = 0.99 \sin(\pi x) \sin(\pi y)$ and 
	$u_0(x,y,z) = 0.99 \sin(\pi x) \sin(\pi y) \sin(\pi z)$  in the two and three dimensions, respectively. 
	The scaling coefficient 
	$0.99$ is chosen to ensure that $ 0 < u_0 < 1$, on $\Omega = (0,1)^d$, \(d = 2, 3\).
\end{example} 
It can be observed that if $0 \le u \le 1 -\varepsilon$ on $\Omega \times (0, T]$, for some $\varepsilon > 0$, then 
\[ 
|f(u)| = \rho \frac{u}{1-u} \le \frac{\rho}{\varepsilon}(1 - \varepsilon) =: c_1(\varepsilon) (1 - \varepsilon), 
\] 
as such the constant \(c_1\) depends on the bound of \(u\). Consequently, the bound on $f$ is not globally uniform.

For testing purposes, we take $\rho = 0.1$ and use spatial grids with $m_x = m_y = 512$ in two dimensions and $m_x = m_y = m_z = 80$ in three
dimensions. Since the exact solution is not available for this problem, the errors $E(N)$ are computed using a reference solution in place of
$u(\mathbf{x},t)$. Specifically, the reference solution is obtained on a fine time grid with $N = 512$ in two dimensions and $N = 320$ in 
three dimensions. Tables~\ref{tab:Unbounded2D} and~\ref{tab:Unbounded3D} report the errors, convergence rates, and CPU times for the 
Pad\'e--ETD2RK-DS scheme in two and three dimensions, with \(T = 1\).  As before, both the LU and Sylvester (spectral) implementations 
yield identical errors, confirming numerical consistency. The optimal $\mathcal{O}(\tau^2)$ temporal convergence rate is again attained, 
demonstrating that the scheme remains stable and accurate even without the global linear-growth assumption.  

Additionally, Figure~\ref{fig:singular_source_prof_2D} displays the computed solution profiles at $y=0.5$ for short ($0\le t\le1$, 
$\tau = 1/128$) and long ($0\le t\le100$, \(\tau=0.1\)) times. The bounded, non-oscillatory evolution confirms the stability of 
the ETD2RK-DS scheme over a long time interval, despite the presence of only locally bounded nonlinear growth.

Further performance comparisons are summarized in Table~\ref{tab:Unbounded_CPU}, where we fix $\tau=1/32$ and vary the spatial resolution 
$h\in\{2^{-i}: i=4,\ldots,11\}$. This is to examine the performance of the different implementation strategies for increasing matrix sizes. 
In addition to the Sylvester-based and sliced LU implementations, we also report results for a sparsity-based LU approach that exploits the 
band structure of the split operators (in the dimension split ETD2RK-DS scheme) without tensor slicing or problem-size reduction (implemented 
using an LU sparse solver). It is observed that while sparsity exploitation improves performance relative to the unsplit formulation, 
the runtime still grows significantly faster than for the sliced implementations. In contrast, the spectral (Sylvester-equation) solver exhibits 
markedly better scalability as the matrix dimension increases, highlighting the advantage of explicit tensor slicing and structure reuse for large systems. 

\begin{table}[] 
	\centering
	\begin{tabular}{lllllll}
		& \multicolumn{3}{c}{Sylvester} & \multicolumn{3}{c}{LU}  \\ 
		\cmidrule(lr){2-4} \cmidrule(lr){5-7}
		$N$ & $E(N)$     & EOC    & CPU    & $E(N)$   & EOC  & CPU   \\ 
		\thickhline
		16  & 2.21E-02  &      & 0.41     & 2.21E-02 &      & 0.66  \\
		32  & 7.17E-03  & 1.63 & 0.77     & 7.17E-03 & 1.63 & 1.30  \\
		64  & 2.20E-03  & 1.70 & 1.49     & 2.20E-03 & 1.70 & 2.60  \\
		128 & 6.30E-04  & 1.80 & 2.96     & 6.30E-04 & 1.80 & 5.22  \\
		256 & 1.46E-04  & 2.11 & 5.93     & 1.46E-04 & 2.11 & 10.51
	\end{tabular} 
	\caption{Example \ref{exp:singular} temporal convergence and CPU-time comparison of the LU and
		Sylvester implementations of the Pad\'e--ETD2RK-DS scheme for the two-dimensional problem
		with locally bounded source.}\label{tab:Unbounded2D}
\end{table} 

\begin{table}[] 
	\centering
	\begin{tabular}{lllllll} 
		& \multicolumn{3}{c}{Sylvester} & \multicolumn{3}{c}{LU}  \\
		\cmidrule(lr){2-4} \cmidrule(lr){5-7}
		$N$ & $E(N)$     & EOC    & CPU    & $E(N)$   & EOC  & CPU   \\ 
		\thickhline
		10  & 1.50E-02    &        & 0.55   & 1.50E-02 &      & 0.77  \\
		20  & 4.78E-03    & 1.65   & 1.06   & 4.78E-03 & 1.65 & 1.54  \\
		40  & 1.39E-03    & 1.78   & 2.15   & 1.39E-03 & 1.78 & 3.07  \\
		80  & 3.67E-04    & 1.92   & 4.28   & 3.67E-04 & 1.92 & 6.10  \\
		160 & 7.84E-05    & 2.23   & 8.63   & 7.84E-05 & 2.23 & 12.36
	\end{tabular} 
	\caption{Example \ref{exp:singular} (3D case), temporal convergence and CPU-time comparison of the LU and
		Sylvester (spectral) implementations of the Pad\'e--ETD2RK-DS scheme for the three-dimensional problem
		with locally bounded source.}\label{tab:Unbounded3D}
\end{table}

\begin{table}[]
	\centering
	\begin{tabular}{lllllll}
		\multicolumn{4}{c}{2D}                           & \multicolumn{3}{c}{3D}        \\ 
		\cmidrule(lr){1-4} \cmidrule(lr){5-7} 
		M    & Spectral & LU (1D-Slice) & LU (Non-Slice) & M   & Spectral & LU(Parallel) \\ 
		\thickhline
		16   & 0.001    & 0.002         & 0.003          & 16  & 0.02     & 0.04         \\
		32   & 0.002    & 0.007         & 0.007          & 32  & 0.05     & 0.23         \\
		64   & 0.012    & 0.016         & 0.027          & 64  & 0.85     & 1.44         \\
		128  & 0.030    & 0.050         & 0.119          & 128 & 6.87     & 9.69         \\
		256  & 0.114    & 0.170         & 0.603          & 256 & 66.16    & 109.88       \\
		512  & 0.771    & 1.309         & 2.625          &     &          &              \\
		1024 & 3.314    & 6.481         & 10.602         &     &          &              \\
		2048 & 16.57    & 37.60         & 51.223         &     &          &   
	\end{tabular}
	\caption{Example \ref{exp:singular}. CPU-time comparison of the Sylvester (spectral) and LU-based
		implementations of the ETD2RK-DS scheme for increasing matrix dimensions, including a
		sparsity-based dimension-split implementation without tensor slicing.}
	\label{tab:Unbounded_CPU}
\end{table}

\begin{figure}
	\centering
	\includegraphics[width=0.49\textwidth]{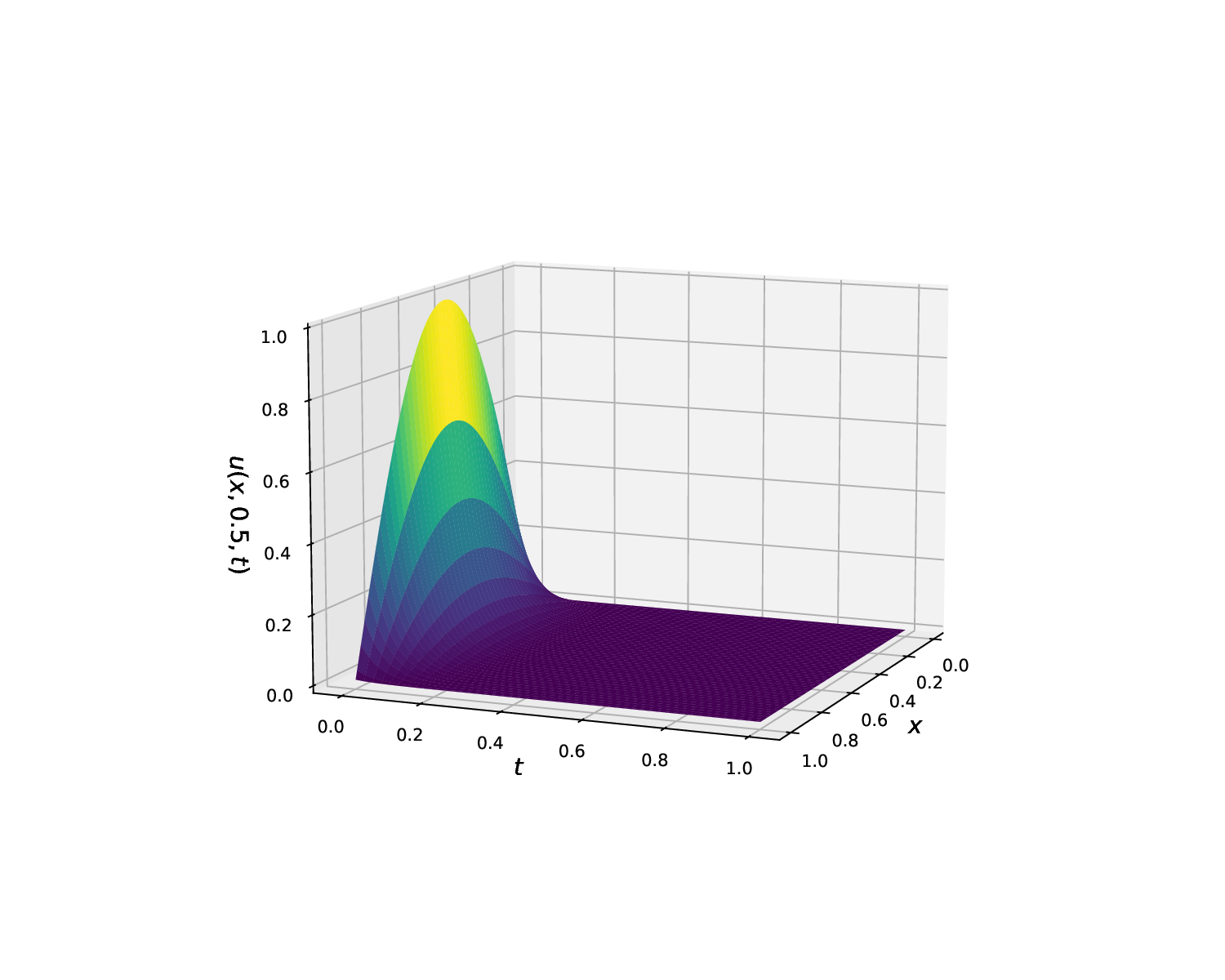} 
	\includegraphics[width=0.49\textwidth]{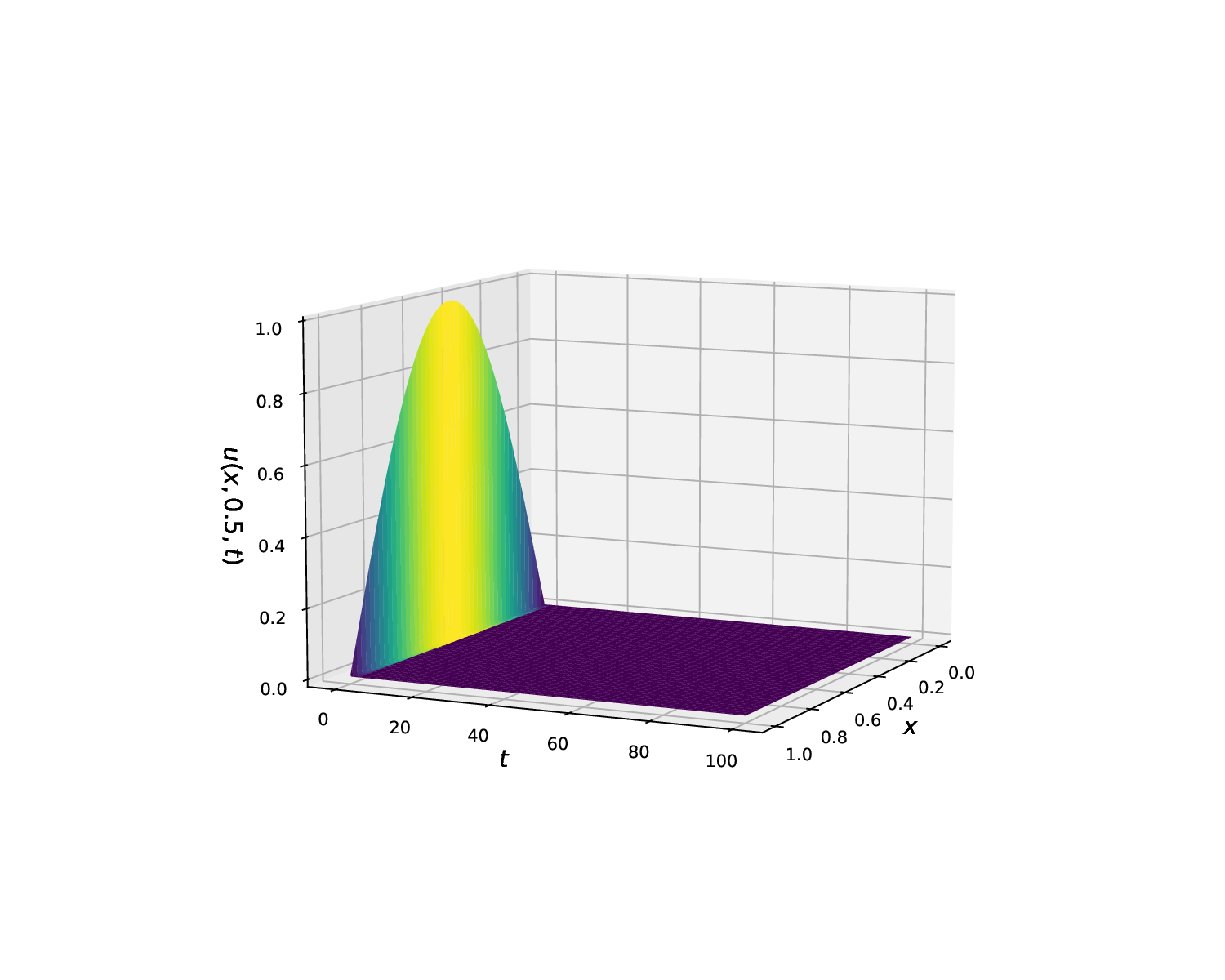}
	\caption{Example \ref{exp:singular}: numerical solution profiles for the two-dimensional problem at $y=0.5$.
		Short-time evolution ($0\le t\le1$, left) and long-time evolution ($0\le t\le100$, right)}
	\label{fig:singular_source_prof_2D}
\end{figure}

\begin{example}[FitzHugh-Nagumo Model (FHN)]\label{exp:FHN} 
	To illustrate the applicability of the ETD2RK-DS scheme to multi-component problems and to highlight some computational 
	benefits of the Sylvester-equation implementation developed in this work, we consider a coupled reaction--diffusion 
	system given by the FitzHugh-Nagumo Model:
	\[ 
	\begin{gathered}
		u_t - \kappa_u \Delta u =  u - \frac{u^3}{3} - v,  
		\nonumber 
		\\
		v_t - \kappa_v \Delta v =  \varepsilon (u -\alpha v), 
	\end{gathered}
	\] 
	posed on the domain $\Omega = (0,1) \times (0,1)$. The problem is completed with homogeneous Dirichlet boundary 
	conditions $u = v = 0$, on $\partial\Omega$ and Gaussian bump $u_0(x,y) = e^{-\frac{(x - 0.5)^2 + (y - 0.5)^2}{\sigma^2}}$ 
	as the initial data for $u$, while for $v$, we take $v_0(x,y) = 0$.  
\end{example} 

Applying the centered difference discretization in Section \ref{sec:discretize} results in the linear system of the 
form \eqref{eq:semidiscrete}, where the coefficient matrix \(A_h = A_1 + A_2\), with the block-diagonal matrices 
\(A_1\) and \(A_2\) given by: 
\[ 
A_1 = \begin{bmatrix}
	I_y \otimes \kappa_u A_x & 0 
	\\ 
	0 & I_y \otimes \kappa_v A_x
\end{bmatrix}, \qquad 
A_2 = \begin{bmatrix}
	\kappa_u A_y \otimes I_x  & 0 
	\\ 
	0 &  \kappa_v A_y \otimes I_x 
\end{bmatrix}, 
\] 
$A_x, A_y, I_x$, and $I_y$ are all as described in Section~\ref{subsec:spatial_Disc}. The dimension-splitting procedure in 
Section~\ref{subsec:2D Dimension Spliting} is then applicable to each of the two uncoupled $(m_x-1)\cdot(m_y-1)$ linear systems
arising in the Pad\'e-based ETD formulation.

An important advantage of the Sylvester-equation formulation, for solving such systems, is that the eigendecompositions of $A_x$ and 
$A_y$ need only be computed once and can be reused for both components, since they do not depend on the diffusivities $\kappa_u$ 
and $\kappa_v$. In contrast, the LU-based implementation must factorize the distinct matrices $\kappa_u A_x-Is$ and $\kappa_v A_x-I\,s$ 
separately whenever $\kappa_u\neq\kappa_v$, with similar factorization repeated for $\kappa_u A_y-I\,s$ and $\kappa_v A_y-I\,s$.  This 
example therefore illustrates the potential computational savings of the Sylvester-based approach for systems with heterogeneous or 
time-dependent diffusivities.

For testing purposes, we choose the parameters $\alpha = \varepsilon = 1$, $\kappa_u = 0.01$, and $\kappa_v = 10$, using uniform spatial 
grids $m_x = m_y = 512$. Further, the errors $E(N)$ are computed using a reference solution obtained on a fine temporal grid with $N=512$. 
The numerical errors and experimental convergence rates for the Sylvester-based implementation are reported in Table~\ref{tab:FHN}; similar
values were obtained with the LU-based implementation (not presented here). The CPU timings, included for comparison, show that the Sylvester method remains 
appreciably more efficient.  In this case, we observe a more pronounced difference in computation time between the LU and Sylvester-based 
implementations compared with the two-dimensional problems in Examples \ref{exp:Allen-Cahn} and \ref{exp:singular}. This likely reflects the 
increased number of matrix factorizations required by the LU approach when the diffusivities satisfy $\kappa_u \neq \kappa_v$.

Plots of the solution components at fixed time $t = 0.125$ are provided in Figure \ref{fig:FHZ_prof_2D}. The activator $u$ exhibits a localized pulse 
that diffuses and decays, while the inhibitor $v$ follows with a smoother response. The resulting Gaussian-like profiles are consistent with the 
diffusive pattern associated with the chosen parameters $(\kappa_u,\kappa_v,\alpha,\varepsilon)=(0.01,10,1,1)$ and the imposed homogeneous Dirichlet 
boundary conditions.

\begin{table}[] 
	\centering
	\begin{tabular}{lllll}
		& \multicolumn{2}{c}{Convergence Results} & \multicolumn{2}{c}{CPU Time} \\ 
		\cmidrule(lr){2-3} \cmidrule(lr){4-5} 
		$N$ & $E(N)$              & EOC             & Sylvester  & LU     \\ 
		\thickhline
		16  & 1.69E-02            &                 & 0.93       & 1.55            \\
		32  & 5.88E-03            & 1.52            & 1.84       & 3.08            \\
		64  & 1.79E-03            & 1.71            & 3.63       & 6.15            \\
		128 & 4.80E-04            & 1.90            & 7.25       & 12.47           \\
		256 & 1.03E-04            & 2.22            & 14.11      & 25.07          
	\end{tabular} 
	\caption{Example \ref{exp:FHN}: Numerical results for the Sylvester based implementation, with $m_x = m_y = 512$ and \(T = 1\). 
		As well as CPU time comparisons of the LU and Sylvester equation based implementations of the ETD2RK-DS scheme for solving 
		the FHN model}\label{tab:FHN}
\end{table} 

\begin{figure}
	\centering
	\includegraphics[width=0.49\textwidth]{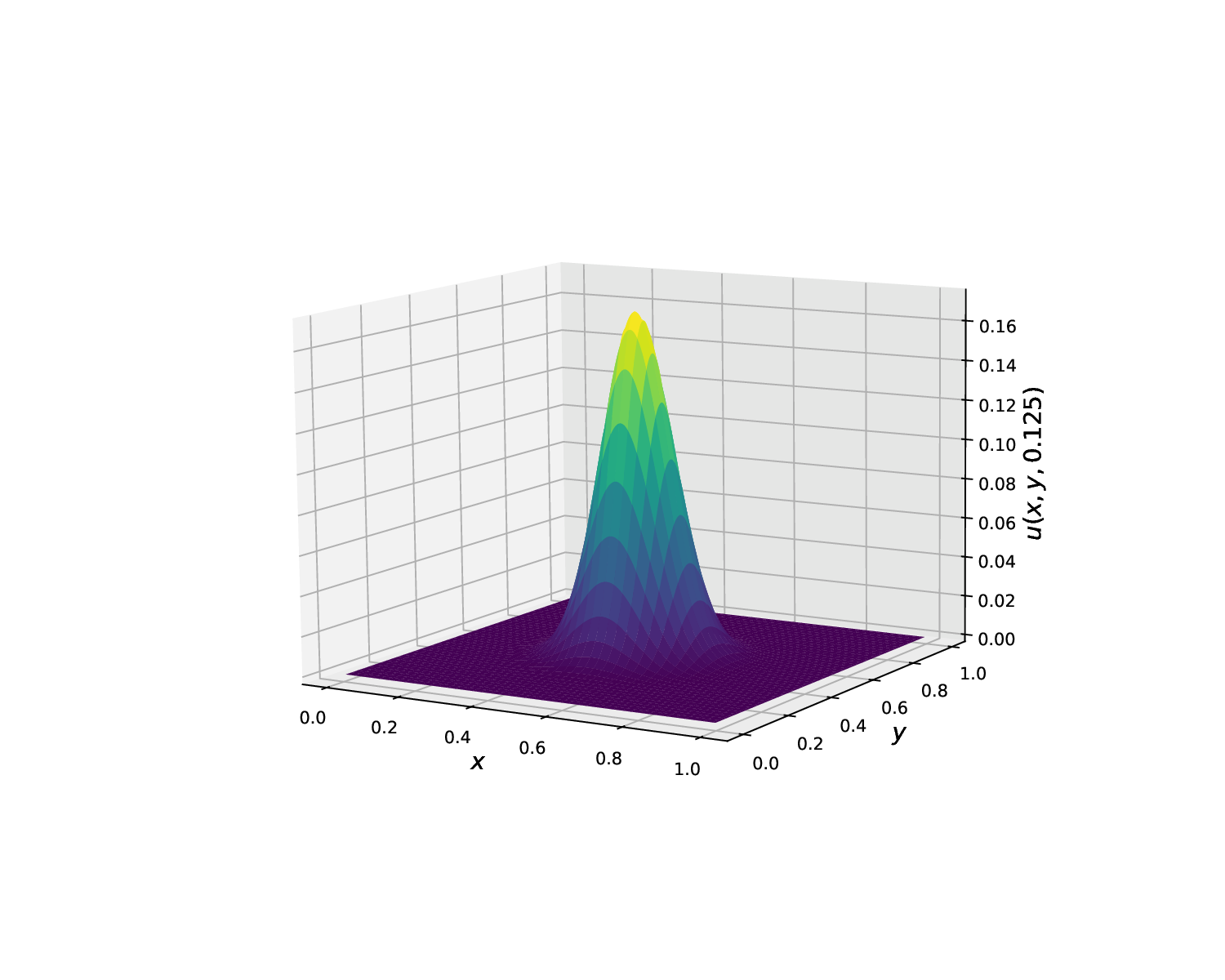} 
	\includegraphics[width=0.49\textwidth]{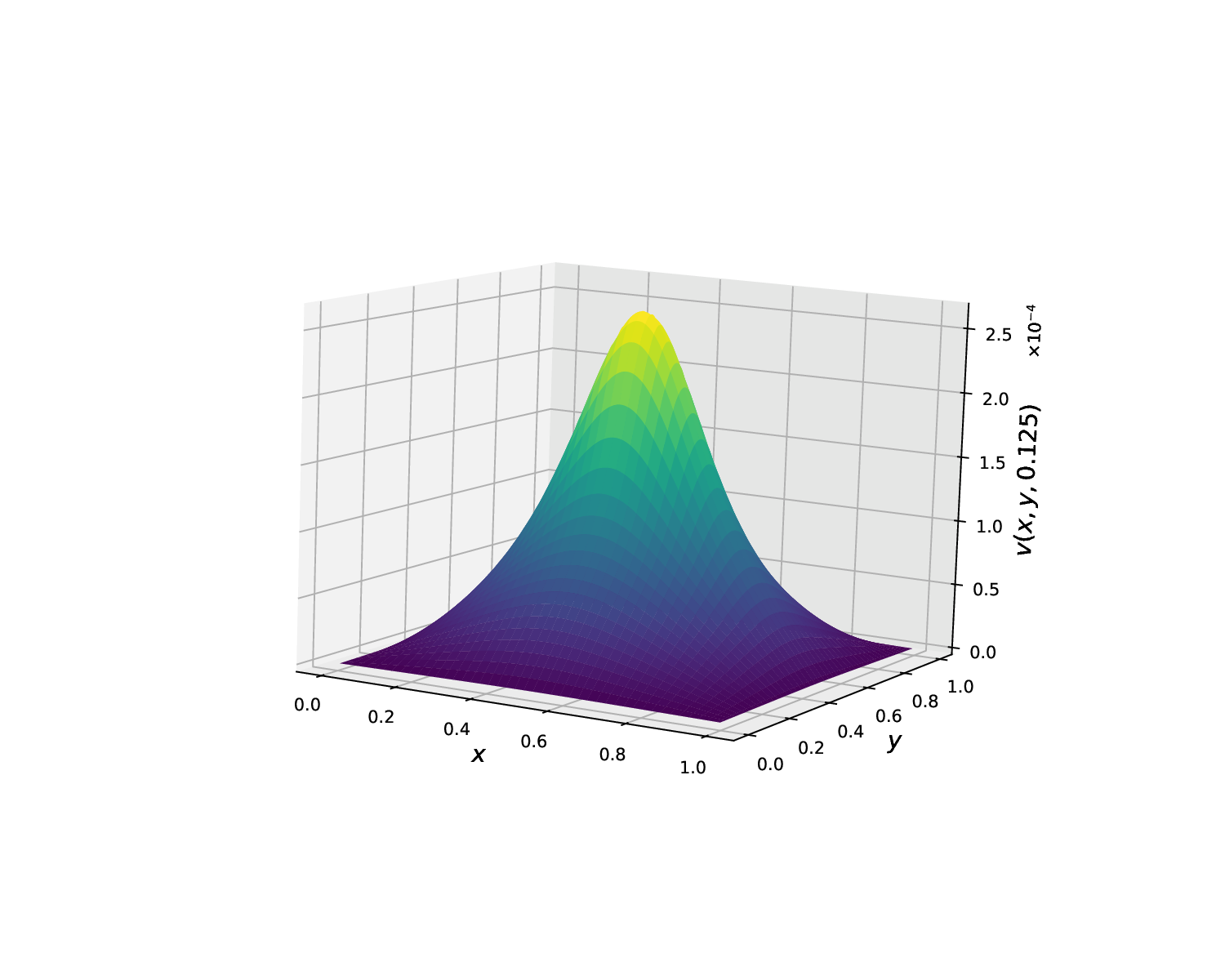}
	\caption{Example \ref{exp:FHN} 3D plot of the solution components $u$ (left) and $v$ (right) at $t = 0.125$, 
		taking $N = 256$}
	\label{fig:FHZ_prof_2D}
\end{figure}  

\section*{Conclusion} 
We have developed and analyzed a second-order, dimension-split exponential time differencing Runge--Kutta scheme 
(ETD2RK--DS) for multidimensional reaction--diffusion equations. For the underlying scheme based on exact matrix 
exponentials, we established uniform stability and second-order temporal convergence under mild assumptions on the 
nonlinear source term, and obtained an efficient fully discrete formulation using Pad\'e approximations, for which a 
complete error analysis was provided. A central contribution of this work is the derivation of explicit, reproducible 
matrix-slicing and rearrangement procedures that realize a genuine algebraic decomposition of the multidimensional 
problem into collections of independent one-dimensional subsystems, reducing the dominant per-time-step computational 
cost from $\mathcal{O}(m^3)$ to $\mathcal{O}(m^2)$ in two dimensions and from $\mathcal{O}(m^5)$ to 
$\mathcal{O}(m^3)$ in three dimensions compared with banded LU solvers applied to the unsplit problem, where $m$ 
denotes the number of grid points per spatial direction. To efficiently solve the resulting shifted one-dimensional 
systems, we introduced a Sylvester-equation reformulation that enables a spectral implementation based on reusable 
eigendecompositions, confining complex-valued calculations associated with Pad\'e approximants to a one-time 
preprocessing step and reducing subsequent linear solves to real-valued matrix--vector multiplications and Hadamard 
divisions. This reformulation yields substantial computational savings relative to LU-based solvers while preserving 
accuracy, with particular advantages for higher-order Pad\'e approximants and problems with heterogeneous diffusivities. 
Numerical experiments in two and three dimensions confirm the second-order temporal convergence of the scheme, 
demonstrate its robustness beyond the global linear-growth setting, and validate the computational advantages of the 
proposed Sylvester-based implementation.

\section*{Data Availability}  
No data was analysed in this study. 

\section*{Acknowledgement} 
The authors would like to acknowledge the support provided by the Deanship of Research 
at King Fahd of Petroleum \& Minerals (KFUPM), Saudi Arabia for funding this work through project No. EC231011. 

\section*{Declarations} 
The authors declare that they have no competing interests. 

\bibliographystyle{elsarticle-num}
\bibliography{mlf,bio,LinearAlgebra,Integer_Order,Porous_Media}  
\end{document}